\documentclass[11pt]{amsart}

\usepackage{amssymb}
\usepackage{bm}
\usepackage[centertags]{amsmath}
\usepackage{amsfonts}
\usepackage{amsthm}
\usepackage{mathrsfs}
\usepackage[usenames]{xcolor}
\usepackage[normalem]{ulem}
\usepackage{comment}

\theoremstyle{definition}
\newtheorem{thm}{Theorem}[section]
\newtheorem{defn}[thm]{Definition}
\newtheorem{lem}[thm]{Lemma}
\newtheorem{prop}[thm]{Proposition}
\newtheorem{cor}[thm]{Corollary}
\newtheorem{rem}[thm]{Remark}

\newtheorem{question}{Question}

\newtheorem*{defn*}{Definition}

\newtheorem*{thm*}{Theorem}

\newtheorem*{cor*}{Corollary}

\newtheorem*{prp*}{Proposition}

\newtheorem{problem}{Problem}

\newtheorem*{ntt}{Notation}

\newtheorem{thmA}{Theorem}

\newtheorem{corA}[thmA]{Corollary}

\newcommand{\inn}{\in\mathbb{N}}

\newcommand{\de}{\delta}

\newcommand{\la}{\lambda}

\newcommand{\Sn}{\mathcal{S}_n}

\newcommand{\e}{\varepsilon}

\newcommand{\N}{\mathbb{N}}

\newcommand{\X}{X_{\mathbf{iw}}}
\newcommand{\W}{W_\mathbf{iw}}

\newcommand{\Xt}{\widetilde X_{\mathbf{iw}}}
\newcommand{\Wt}{\widetilde W_\mathbf{iw}}

\newcommand{\cose}{\makebox[-1.6pt][s]{\raisebox{-0.03ex}{$\not$}}\makebox[0.6pt][s]{\raisebox{0.2ex}{$\Rightarrow$}}\makebox[9pt][s]{\raisebox{-0.4ex}{\Tiny$\hookleftarrow$}}\;}

\newcommand{\iii}[1]{{\left\vert\kern-0.25ex\left\vert\kern-0.25ex\left\vert #1 
    \right\vert\kern-0.25ex\right\vert\kern-0.25ex\right\vert}}

\newcommand\restr[2]{{
  \left.\kern-\nulldelimiterspace 
  #1 
  \vphantom{\big|} 
  \right|_{#2} 
  }}

\DeclareMathOperator{\supp}{supp}

\DeclareMathOperator{\ran}{ran}

\long\def\symbolfootnote[#1]#2{\begingroup%
\def\thefootnote{\fnsymbol{footnote}}\footnote[#1]{#2}\endgroup}

\allowdisplaybreaks

\begin{document}

\title[Complete separation of asymptotic structures]{On the complete separation of asymptotic structures in Banach spaces}

\author[S. A. Argyros]{Spiros A. Argyros}
\address{National Technical University of Athens, Faculty of Applied Sciences,
Department of Mathematics, Zografou Campus, 157 80, Athens, Greece}
\email{sargyros@math.ntua.gr}

\author[P. Motakis]{Pavlos Motakis}
\address{Department of Mathematics, University of Illinois at Urbana-Champaign, Urbana, IL 61801, U.S.A.}
\email{pmotakis@illinois.edu}


\thanks{The second named author
  was  supported by the National Science Foundation under Grant Numbers
  DMS-1600600 and DMS-1912897.}

\begin{abstract}
Let $(e_i)_i$ denote the unit vector basis of $\ell_p$, $1\leq p< \infty$, or $c_0$. We construct a reflexive Banach space with an unconditional basis that admits $(e_i)_i$ as a uniformly unique spreading model while it has no subspace with a unique asymptotic model, and hence it has no asymptotic-$\ell_p$ or $c_0$ subspace. This solves a problem of E. Odell. We also construct a space with a unique $\ell_1$ spreading model and no subspace with a uniformly unique $\ell_1$ spreading model. These results are achieved with the utilization of a new version of the method of saturation under constraints that uses sequences of functionals with increasing weights.
\end{abstract}

\maketitle

\setcounter{tocdepth}{1}
\tableofcontents

\section{Introduction}
The study of asymptotic properties lies at the heart of Banach space theory. It is intertwined with other central notions of Banach spaces, e.g., distortion, bounded linear operators, and metric embeddings. There exists a wide plethora of examples that demonstrate deep connections between each of the aforementioned topics and asymptotic properties. A Banach space that is boundedly distortable must contain an asymptotic-$\ell_p$ subspace \cite{MT}, properties of spreading models can be manipulated to construct reflexive Banach spaces on which every bounded linear operator has a non-trivial closed invariant subspace \cite{AM1}, and reflexive asymptotic-$c_0$ spaces provide the first known class of Banach spaces into which there is no coarse embedding of the Hilbert space \cite{BLS}. There exists plenty of motivation to further understand asymptotic notions and to work on problems in the theory defined by them. It is highly likely that such understanding may play a crucial role in solving open problems in other branches of the theory.

One of the main goals of this article is to answer an old open problem regarding the relationship between spreading models and asymptotic-$\ell_p$ spaces: if $X$ admits a unique spreading model with a uniform constant, must $X$ contain an asymptotic-$\ell_p$ subspace? It was first formulated by E. Odell in \cite{O1} and it was reiterated in \cite{O} as well as in \cite{JKO}. We construct a Banach space $\X$ that serves as a counterexample to this question. At the same time it reveals information regarding the relationship between asymptotic properties at a deeper level than the one suggested by the question of Odell. A property (P) of Banach spaces is called hereditary if whenever $X$ has (P) then all of its infinite dimensional closed subspaces have (P) as well. We discuss two degrees in which two asymptotic, and more generally hereditary, properties of Banach spaces can be distinct. 
\begin{defn*}
\label{property separation}

Let (P) and (Q) be two hereditary properties of Banach spaces and assume that (P) implies (Q).

\begin{itemize}

\item[(i)] If (Q)$\not\Rightarrow$(P), i.e., there exists a Banach space $X$ satisfying (Q) and failing (P) then we say that (P) is separated from (Q).

\item[(ii)] If there exists a Banach space $X$ satisfying (Q) and every infinite dimensional closed subspace $Y$ of $X$ fails (P) then we way that (P) is completely separated from (Q) and write (Q) \cose (P).

\end{itemize}
\end{defn*}

For example, if (P) is super-reflexivity and (Q) is reflexivity then (Q) \cose (P). Indeed, Tsirelson space from \cite{T} is reflexive, yet it contains no super-reflexive subspaces. In this paper we mainly consider properties that are classified into the following three categories: the sequential asymptotic properties, the array asymptotic properties, and the global asymptotic properties. For expository purposes in this introduction we shall only consider reflexive Banach spaces with a basis and block sequences of vectors, although these are in general not necessary restrictions. More details on this can be found in Section \ref{asymptotic structures section}. 

Sequential asymptotic properties are related to the spreading models generated by sequences in a space. Recall that a spreading model is a concept that describes the asymptotic behavior of a single sequence $(x_j)_j$ in a Banach space. It was introduced in \cite{BS} and it has been an integral part of Banach space theory ever since. We say that a Banach space has a unique block spreading model if any two spreading models generated by normalized block sequences in $X$ are equivalent and we say that $X$ has a uniformly unique block spreading model if the same as before holds with the additional assumption that the equivalence occurs for a uniform $C$. By the proof of Krivine's theorem from \cite{K}, uniform uniqueness of a spreading model implies that it has to be equivalent to the unit vector basis of $\ell_p$, for some $1\leq p<\infty$, or $c_0$.

The array asymptotic properties concern the asymptotic behavior of arrays of sequences $(x_j^{(i)})_j$, $i\in\N$, in a space. Two tools used for this purpose are the asymptotic models and the joint spreading models introduced in \cite{HO} and \cite{AGLM} respectively. Uniqueness of these notions is defined in a similar manner to uniform uniqueness of spreading models. They were used in \cite{BLMS} to show that the class of reflexive asymptotic-$c_0$ Banach spaces is coarsely rigid and in \cite{AGLM} to show that whenever a Banach space has a unique joint spreading model then it satisfies a property concerning its space of bounded linear operators, called the UALS. Although asymptotic models and joint spreading models are not identical they are strongly related. A Banach space has a unique block asymptotic model if and only if it has a unique block joint spreading model and then it has to be equivalent to the unit vector basis of $\ell_p$, for some $1\leq p<\infty$, or $c_0$. Another concept related to array asymptotic properties is that of asymptotically symmetric spaces from \cite{JKO}.

Global asymptotic properties, roughly speaking, describe the behavior of finite block sequences $(x_i)_{i=1}^n$ that are chosen sufficiently far apart in a space $X$ with a basis. If these exist $C\geq 1$ so that for all $n\in\N$, all normalized block sequences $(x_i)_{i=1}^n$ with support after $n$, are $C$-equivalent to one another, then it follows that they all have to be uniformly equivalent to the unit vector basis of $\ell_p^n$, for some $1\leq p\leq \infty$ and we say that $X$ is an asymptotic-$\ell_p$ space (or an asymptotic-$c_0$ space if $p=\infty$). This concept was introduced in  \cite{MT} and it was generalized in \cite{MMT} to a coordinate free version for spaces with or without a basis. Given a Banach space $X$ with a basis we will mainly focus on the properties in the following list. Here, $1\leq p\leq \infty$ and whenever $p=\infty$ then $\ell_p$ should be replaced with $c_0$.

\begin{itemize}
\item[(a)$_p$] The space $X$ is asymptotic-$\ell_p$.
\item[(b)$_p$] The space $X$ admits a unique $\ell_p$  block asymptotic model.
\item[(c)$_p$] The space $X$ admits a uniformly unique  $\ell_p$ block spreading model.
\item[(d)$_p$] The space $X$ admits a unique  $\ell_p$ block spreading model.
\end{itemize}
Given the precise definitions, which will be provided in Section \ref{asymptotic structures section}, the following implications are fairly straightforward for all $1\leq p\leq \infty$: (a)$_p\Rightarrow$ (b)$_p\Rightarrow$ (c)$_p\Rightarrow$ (d)$_p$. Whether the corresponding converse implications hold depends on $p$.  In the case $1\leq p<\infty$ none of them is true: (d)$_p\not\Rightarrow$ (c)$_p$, $1\leq p<\infty$ is easy whereas (c)$_p\not\Rightarrow$ (b)$_p$, $1\leq p<\infty$ and (b)$_p\not\Rightarrow$ (a)$_p$, $1<p<\infty$  were shown in \cite{BLMS}. It was also shown in that paper that (c)$_\infty\not\Rightarrow$ (b)$_\infty$ and in \cite{AGM} it was shown that (b)$_1\not\Rightarrow$ (a)$_1$. However, it was proved in \cite{AOST} that (c)$_\infty\Leftrightarrow$ (d)$_\infty$ and a remarkable recent result from \cite{FOSZ} states that (b)$_\infty\Leftrightarrow$ (a)$_\infty$. This last result requires the coordinate free definition of asymptotic-$\ell_p$ from \cite{MMT}.

The problem of Odell that was mentioned earlier in the introduction can be formulated in the language of this paper as follows: is there $1\leq p\leq \infty$ so that (c)$_p$ \cose (a)$_p$? We actually prove something deeper, namely that (c)$_p$ \cose (b)$_p$ for all $1\leq p\leq \infty$. We also prove (d)$_1$ \cose (c)$_1$, although the same argument works for $1<p<\infty$ (as it was mentioned earlier (c)$_\infty\Leftrightarrow$ (d)$_\infty$). To achieve these results we present three constructions of Banach spaces. Let us describe the properties of these spaces one by one and later give an outline of how they are defined. The first construction yields (c)$_1$ \cose (b)$_1$.

\begin{thmA}
\label{ell1 case main theorem}

There exists a reflexive Banach space $\X$  that has a $1$-unconditional basis and the following properties:

\begin{itemize}

\item[(i)] every normalized weakly null sequence in $\X$ has a subsequence that generates a spreading model that is $4$-equivalent to the unit vector basis of $\ell_1$.

\item[(ii)] every infinite dimensional subspace of $\X$ contains an array of normalized weakly null sequences that generate the unit vector basis of $c_0$ as an asymptotic model.

\end{itemize}
That is, (c)$_1$ \cose (b)$_1$ and in particular (c)$_1$ \cose (a)$_1$. Additionally,

\begin{itemize}

\item[(iii)] every finite dimensional Banach space with a $1$-unconditional basis is finitely block representable in every block subspace of $\X$. More precisely, it is an asymptotic space of every infinite dimensional subspace of $\X$.

\end{itemize}

\end{thmA}
The third property was first shown to be satisfied by a space constructed by Odell and Th. Schlumprecht in \cite{OS1}. It yields that the set $[1,\infty]$ is a stable Krivine set of $\X$, i.e., it is a Krivine set of every block subspace of $\X$. The second construction is a variation of the first one and it yields  (c)$_p$ \cose (b)$_p$, $1<p<\infty$.

\begin{thmA}
\label{ellp case main theorem}

For every $1<p<\infty$ there exists a reflexive Banach space with a $1$-unconditional basis that has the following properties.

\begin{itemize}

\item[(i)] Every normalized weakly null sequence in $\X^p$ has a subsequence that generates a spreading model that is $8$-equivalent to the unit vector basis of $\ell_p$.

\item[(ii)] Every infinite dimensional subspace of $\X^p$ contains an array of normalized weakly null sequences that generate the unit vector basis of $c_0$ as an asymptotic model.

\end{itemize}
That is, (c)$_p$ \cose (b)$_p$ and in particular (c)$_p$ \cose (a)$_p$. Additionally,

\begin{itemize}

\item[(iii)] For every $1\leq q \leq \infty$ and block subspace $Y$ of $\X^p$ the unit vector basis of $\ell_q$ is finitely block representable in $Y$ if and only if $p\leq q\leq\infty$. More precisely, for $p\leq q\leq \infty$ and $n
\in\N$, $\ell_q^n$ is an asymptotic space of every infinite dimensional subspace of $\X^p$.

\end{itemize}

\end{thmA}
Property (iii) resembles the corresponding property of Theorem \ref{ell1 case main theorem}. We point out that for $1<p<\infty$ the space $\X^p$ is the first known example of a space with $[p,\infty]$ as a stable Krivine set. Recall that in \cite{BFM} for every discrete closed subset $F$ of $[1,\infty]$ a space is constructed with stable Krivine set $F$. The fact (c)$_\infty$ \cose (b)$_\infty$ is not achieved via a separate construction.

\begin{thmA}
\label{c0 case main theorem}

The space $\X^*$ has the following properties.

\begin{itemize}

\item[(i)] Every normalized weakly null sequence has a subsequence that generates a spreading model that is $4$-equivalent to the unit vector basis of $c_0$.

\item[(ii)] Every infinite dimensional subspace of $\X^*$ contains an array of normalized weakly null sequences that generate the unit vector basis of $\ell_1$ as an asymptotic model.

\end{itemize}
That is, (c)$_\infty$ \cose (b)$_\infty$ and in particular (c)$_\infty$ \cose (a)$_\infty$.
\end{thmA}

We additomnally observe that the spaces $\X$ and $\X^*$ are asymptotically symmetric and obtain a  negative answer to \cite[Problem 0.2]{JKO}.

\begin{corA}
There exist Banach spaces that are asymptotically symmetric and have no asympotic-$\ell_p$ or $c_0$ subspaces.
\end{corA}

A stronger version of the above corollary was obtained in \cite{KM} where it was shown that there exists an asymptotically symmetric Banach space with no subspace that admits a unique spreading model. The  final construction yields (d)$_1$ \cose (c)$_1$.

\begin{thmA}
\label{cd separation main theorem}

There exists a reflexive Banach space $\Xt$ that has a $1$-unconditional basis and the following properties.

\begin{itemize}

\item[(i)] Every normalized weakly null sequence has a subsequence that generates a spreading model that is equivalent to the unit vector basis of $\ell_1$.

\item[(ii)] In every infinite dimensional subspace of $\Xt$ and for every $C\geq 1$ there exists a normalized weakly null sequence that generates a spreading model that is not $C$-equivalent to the unit vector basis of $\ell_1$.

\end{itemize}
That is, (d)$_1$ \cose (c)$_1$.

\end{thmA}
It is also possible to construct for each $1<p<\infty$ a variation $\Xt^p$ of $\Xt$ that yields (d)$_p$ \cose (c)$_p$. In contrast to $\X^*$, the space $\Xt^*$ does not have a unique $c_0$ spreading model.

Each of the aforementioned spaces are constructed with the use of a saturated norming set. We use the general scheme of saturation under constraints, which was first used in \cite{OS1} and \cite{OS2} and later refined in \cite{AM1}, \cite{ABM}, and others. In these aforementioned papers use Tsirelson-type constructions in which functionals in the norming set can only be constructed using very fast growing sequences of averages of elements in the same norming set. We shall refer to this particular version of the scheme as saturation under constraints with growing averages. In this paper we introduce a method that we call saturation under constraints with increasing weights. In this method the construction of functionals in the norming set is allowed only using sequences of functionals from the same norming set that have weights that increase sufficiently rapidly. The mainframe for the norming set $\W$ of $\X$ is the mixed-Tsirelson norming set  $W = W(1/m_j,\mathcal{S}_{n_j})_{j\in\N}$ for appropriate increasing sequences of natural numbers $(m_j)_j$ and $(n_j)_j$. This is the smallest symmetric subset of $c_{00}(\N)$ containing the unit vector basis and so that for all $\mathcal{S}_{n_j}$-admissible (see Subsection \ref{subsection Schreier sets}) elements $f_1<\cdots<f_d$ of $W$ the element $f = (1/m_j)\sum_{q=1}^df_q$ is also in $W$. The weight of such an $f$ is $w(f) = m_j$. In other words, $W$ is closed under the $(1/m_j,\mathcal{S}_{n_j})$-operations. It follows that if we take $i_1,\ldots,i_k$ in $\N$ then the set $W$ is closed under the $(1/(m_{i_1}\cdots m_{i_k}),\mathcal{S}_{n_{j_1}+\cdots +n_{j_k}})$-operation. The set $\W$ is defined to be the smallest subset of $W$ that is closed under the operations $(1/(m_{i_1}\cdots m_{i_k}),\mathcal{S}_{n_{j_1}+\cdots+n_{j_k}})$ applied only to sequences the weights of which increase sufficiently rapidly, i.e. their weights are very fast growing. Consequently, every functional $f\in\W$ is a weighted sum of the form
\begin{equation}
\label{operation equation}
f = \frac{1}{m_{i_1}\cdots m_{i_k}}\sum_{q=1}^d f_q,
\end{equation}
where $(f_q)_{q=1}^d$ is an $\mathcal{S}_{n_{j_1}+\cdots +n_{j_k}}$-admissible sequence of functionals in $\W$ with very fast growing weights. The weight of such an $f$ is $w(f) = m_{i_1}\cdots m_{i_k}$.

The constraint applied to weights of functionals instead of sizes of averages yields relatively easily that the space has a unique $\ell_1$ spreading model whereas including all $(1/(m_{i_1}\cdots m_{i_k}),\mathcal{S}_{n_{j_1}+\cdots +n_{j_k}})$-operations makes this spreading model uniform. With some work it is then shown that finite arrays of sequences of so-called exact vectors with appropriate weights generate an asymptotic model equivalent to the unit vector basis of $c_0$. The proof of this uses a basic inequality where the auxiliary space is also defined with the use of constraints. The spaces $\X^p$, $1<p<\infty$, are defined along the same lines with the difference being that in \eqref{operation equation} the functionals $f_q$ are multiplied by coefficients in the unit ball of $\ell_{p'}$, where $1/p+1/p'=1$. The proof of Theorem \ref{c0 case main theorem} is fundamentally different from the other cases. The fact that $\X^*$ admits a uniformly unique block $c_0$ spreading model is shown directly for elements in the convex hull of $\W$ by manipulating the definition of the norming set. However, the fact that $\X^*$ is rich with arrays of sequences that generate an $\ell_1$ asymptotic model uses some of the structural properties of $\X$.

The norming set of the space $\Wt$ is simpler that $\W$. It is the the smallest subset of $W$ that is closed under the operations $(1/m_{j},\mathcal{S}_{n_{j}})$ applied to very fast growing sequences of weighted functionals. This means that this norming set is closed under fewer operations and hence it is a subset of $\W$. The result is that the space admits only $\ell_1$ spreading models, albeit with arbitrarily bad equivalence constants in every subspace.

\section{Preliminaries}\label{preliminary section}
We remind basic notions such as Schreier families and special convex combinations. Given two non-empty subsets of the natural numbers $A$ and $B$ we shall write $A<B$ if $\max(A)<\min(B)$ and given $n\in\N$ we write $n \leq A$ if $n\leq \min(A)$. We also make the convention $\emptyset <A$ and $A < \emptyset$ for all $A\subset \N$. We denote by $c_{00}(\N)$ the space of all real valued sequences $(c_i)_i$ with finitely many non-zero entries. We denote by $(e_i)_i$ the unit vector basis of $c_{00}(\N)$. In some cases we shall denote it by $(e_i^*)_i$. For $x = (c_i)_i\in c_{00}(\N)$, the support of $x$ is defined to be the set $\supp(x) = \{i\inn:\;c_i\neq 0\}$ and the range of $x$, denoted by $\ran(x)$, is defined to be the smallest interval of $\N$ containing $\supp(x)$. We say that the vectors $x_1,\ldots,x_k$ in $c_{00}(\N)$ are successive if $\supp(x_i) < \supp(x_{i+1})$ for $i=1,\ldots,k-1$. In this case we write $x_1<\cdots<x_k$. Given $n\in\N$ and $x\in c_{00}(\N)$ we also write $n \leq x$ if $n\leq \min\supp(x)$. A (finite or infinite) sequence of successive  vectors in $c_{00}(\N)$ is called a block sequence.

\subsection{Schreier sets}
\label{subsection Schreier sets}
 The Schreier families form an increasing sequence of families of finite subsets of the natural numbers, which first appeared in \cite{AA}. It is inductively defined in the
following manner. Set
\begin{equation*}
\mathcal{S}_0 = \big\{\{i\}: i\inn\big\}\;\text{and}\;\mathcal{S}_1 = \{F\subset\mathbb{N}: \#F\leqslant\min(F)\}
\end{equation*}
and if $\Sn$ has been defined and set
\begin{equation*}
\begin{split}
\mathcal{S}_{n+1} = \left\{\vphantom{\cup_{i = 1}^d}\right.&F\subset\mathbb{N}:\; F = \cup_{i = 1}^d F_i, \;\text{where}\; F_1 <\cdots< F_d\in\Sn\\
&\left.\vphantom{\cup_{j = 1}^k}\text{and}\; d\leqslant\min(F_1)\right\}.
\end{split}
\end{equation*}
For each $n$, $\Sn$ is a regular family. This means that it is hereditary, i.e. if $F\in\Sn$ and $G\subset F$ then $G\in\Sn$, it is spreading, i.e. if $F = \{i_1<\cdots<i_d\} \in\Sn$ and $G = \{j_1 < \cdots < j_d\}$ with $i_p \leqslant j_p$ for $p=1,\ldots,d$, then $G\in\Sn$ and finally it is compact, if seen as a subset of $\{0,1\}^\N$. For each $n\in\N$ we also define the regular family  $$\mathcal{A}_n = \{F\subset \N: \#F\leq n\}.$$ For arbitrary regular families $\mathcal{A}$ and $\mathcal{B}$ we define
\begin{equation*}
\begin{split}
\mathcal{A}*\mathcal{B} = \left\{\vphantom{\cup_{i = 1}^d}\right.&F\subset\mathbb{N}: F = \cup_{i = d}^k F_i,\; \mbox{where}\; F_1 <\cdots< F_d\in\mathcal{B}\\
&\left.\vphantom{\cup_{j = 1}^k}\mbox{and}\; \{\min(F_i): i=1,\ldots,d\}\in\mathcal{A}\right\},
\end{split}
\end{equation*}
then it is well known \cite{AD2} and follows easily by induction that $\Sn*\mathcal{S}_m = \mathcal{S}_{n+m}$. Of particular interest to us is the family $\mathcal{S}_n\ast\mathcal{A}_m$, that is the family of all sets of the form $F = \cup_{i=1}^dF_i$ with $F_1<\cdots< F_d$ with $\#F_i\leq m$ for $1\leq i\leq d$ and $\{\min(F_i):1\leq i\leq d\}\in\mathcal{S}_n$. From the spreading property of $\mathcal{S}_n$ it easily follows that such an $F$ is the union at most $m$ sets in $\mathcal{S}_n$. Given a regular family $\mathcal{A}$ a sequence of vectors $x_1 <\cdots<x_k$ in $c_{00}(\N)$ is said to be $\mathcal{A}$-admissible if $\{\min\supp(x_i): i=1,\ldots,k\}\in\mathcal{A}$.

\subsection{Special convex combinations}
The reading of this subsection may be postponed until before Section \ref{section aux}. Here, we remind the notion of the $(n,\e)$ special convex combinations, (see \cite{AD2},\cite{AGR},\cite{AT}).

\begin{defn}\label{def of basic scc}
Let $x = \sum_{i\in F}c_ie_i$ be a vector in $c_{00}(\N)$, $n\inn$, and $\e>0$. The vector $x$ is called a $(n,\e)$-basic special convex combination (or a $(n,\e)$-basic s.c.c.) if the following are satisfied:

\begin{enumerate}

\item[(i)] $F\in\Sn$, $c_i\geqslant 0$ for $i\in F$ and $\sum_{i\in F}c_i = 1$,

\item[(ii)] for any $G\subset F$ with $G\in\mathcal{S}_{n-1}$ we have that $\sum_{i\in G}c_i < \e$.

\end{enumerate}
\end{defn}

The next result is from \cite{AMT}. For a proof see \cite[Chapter 2, Proposition 2.3]{AT}.

\begin{prop}\label{basic scc exist in abundance}
For every infinite subset of the natural numbers $M$, any $n\inn$, and $\e>0$ there exist $F\subset M$ and non-negative real numbers $(c_i)_{i\in F}$ so that the vector $x = \sum_{i\in F}c_ie_i$ is a $(n,\e)$-basic s.c.c.
\end{prop}

\begin{defn}\label{def scc}
Let $x_1 <\cdots<x_d$ be vectors in $c_{00}(\N)$ and $\psi(i) = \min\supp (x_i)$, for $i=1,\ldots,d$. If the vector $\sum_{i=1}^mc_ie_{\psi(i)}$ is a $(n,\e)$-basic s.c.c. for some $n\inn$ and $\e>0$ then the vector $x = \sum_{i=1}^mc_ix_i$ is called a $(n,\e)$-special convex combination (or $(n,\e)$-s.c.c.).
\end{defn}

We make a few simple remarks to be used in the sequel.
\begin{rem}
\label{some remarks for the far future 1}
Let $n\in\N$, $\e>0$, and $x = \sum_{i\in F}c_ie_i$ be a $(n,\e)$ special convex combination. If $k,m\in\N$ with $k<n$ and $G\subset F$ with $G\in\mathcal{S}_k\ast\mathcal{A}_m$ then $\sum_{i\in G}c_i < m\e$.
\end{rem}

\begin{rem}
\label{some remarks for the far future 2}
Let $n\in\N$, $\e>0$, and $x = \sum_{i\in F}c_ie_i$ be a $(n,\e)$ special convex combination. If $F = \{t_1<\cdots<t_d\}$ we can write $x = \sum_{i=1}^d \tilde c_i e_{t_i}$. If $G\subset \N$ is of the form $G = \{s_1<\cdots<s_d\}$ with $t_i\leq s_i$ for $1\leq i\leq d$ and $s_i\leq t_{i+1}$ for $1\leq i<d$ then the vector $x = \sum_{i=1}^d\tilde c_i e_{s_i}$ is a $(n,2\e)$ special convex combination. In particular, if $x = \sum_{i=1}^mc_ix_i$ is a $(n,\e)$-s.c.c. and $\phi(i) = \max\supp(x_i)$ for $1\leq i\leq d$ then the vector $\sum_{i=1}^dc_ie_{\phi(i)}$ is a $(n,2\e)$-basic s.c.c.
\end{rem}

\section{Asymptotic structures}
\label{asymptotic structures section}
In this lengthy section we remind, compare, and discuss different types of asymptotic notions in Banach space theory. We state known examples that separate these notions in various ways and we discuss how the present paper is an advancement in this topic.

\subsection{Sequential asymptotic notions}
We remind the definition of spreading models, which was introduced in \cite{BS}.

\begin{defn}
Let $(x_i)_i$ be a sequence in a seminormed vector space $(E,\iii{\cdot})$. and $m\in\mathbb{N}$
\begin{itemize}

\item[(i)] A set $s = \{j_1,\ldots,j_m\}\in[\mathbb{N}]$ will be called a spread of $I = \{1,\ldots,m\}$.
\item[(ii)] If $x = \sum_{i=1}^ma_ix_i$ and $s = \{j_1,\ldots,j_m\}$ is a spread of $\{1,\ldots,m\}$ then we call the vector $s(x) = \sum_{i=1}^ma_ix_{j_i}$ a spread of the vector $x$.
\item[(iii)] The sequence $(x_i)_i$ will be called spreading if for every $m\in\mathbb{N}$, every $s\in[\mathbb{N}]^m$, and every $x = \sum_{i=1}^ma_ix_i$ we have $\iii{x} = \iii{s(x)}$.
\end{itemize}
\end{defn}

\begin{defn}
\label{definition spreading model}
Let $X$ be a Banach space and $(x_i)_i$ be a sequence in $X$. Let also $E$ be a vector space with a Hamel basis $(e_i)_i$ endowed with a seminorm $\iii{\cdot}$. We say that the sequence $(x_i)_i$ generates $(e_i)_i$ as a spreading model if for every $m\in\mathbb{N}$ and any vector $x = \sum_{i=1}^ma_ix_i$ we have
\[\lim_{\substack{\min(s)\to\infty\\s\in[\mathbb{N}]^m}}\|s(x)\| = \iii{\sum_{i=1}^ma_ie_i}.\]
Given a subset $A$ of $X$ we shall say that $A$ admits $(e_i)_i$ as a spreading model if there exists a sequence in $A$ that generates $(e_i)_i$ as a spreading model.
\end{defn}

The spreading model $(e_i)_i$ of a sequence $(x_i)_i$ is always a spreading sequence. The above definition was given by Brunel and Sucheston in \cite{BS} where is was also proved that every bounded sequence in a Banach space has a subsequence that generates some spreading model.

\subsection{Array asymptotic notions}
We remind the notion of joint spreading models from \cite{AGLM} and the one of asymptotic models from \cite{HO}. We compare these similar notions later in Subsection \ref{uniqueness sec}.
\begin{defn}
\label{plegma}
Let $k$, $l\in\mathbb{N}$, and $M\in[\mathbb{N}]^\infty$. A plegma is a sequence $(s_i)_{i=1}^l$ in $[M]^k$ satisfying
\begin{itemize}
\item[(i)] $s_{i_1}(j_1) < s_{i_2}(j_2)$ for $i_1\neq i_2$ in $\{1,\ldots,l\}$ and $j_1<j_2$  in $\{1,\ldots,k\}$ and
\item[(ii)] $s_{i_1}(j) \leq s_{i_2}(j)$ for $i_1<i_2$ in $\{1,\ldots,l\}$ and $j\in\{1,\ldots,k\}$.
\end{itemize}
If additionally the set $s_1,\ldots,s_l$ are pairwise disjoint then we say that  $(s_i)_{i=1}^l$ is a strict plegma. Let $\mathrm{Plm}_l([M]^k)$ denote the collection of all plegmas in $[M]^k$ and let $\mathrm{S}$-$\mathrm{Plm}_l([M]^k)$ denote the collection of all strict plegmas in $[M]^k$.
\end{defn}

A plegma $(s_i)_{i=1}^l$ can also be described as follows
\[
\begin{split}
s_1(1) \leq s_2(1) \leq \cdots \leq s_l(1)&<s_1(2)\leq s_2(2)\leq\cdots\leq s_l(2)<\cdots\\
\cdots& <s_1(k)\leq s_2(k)\leq \cdots\leq s_l(k)
\end{split}
\]
whereas in a strict plegma all inequalities are strict.

\begin{defn}
\label{plegma spreading}
Let $l\in\mathbb{N}$ and $(x^{(i)}_j)_j$, $1\leq i\leq l$,   be an array of sequences in a seminormed vector space $(E,\iii{\cdot})$.
\begin{itemize}

\item[(i)] For $m\in\mathbb{N}$ let $\pi = \{1,\ldots,l\}\times\{1,\ldots,m\}$. Given a plegma $\bar s = (s_i)_{i=1}^l$ in $[M]^\infty$, the set $\bar s(\pi) = \{(i,s_i(j)): (i,j)\in\pi\}$ will be called a plegma shift of $\pi$.
\item[(ii)] If $x = \sum_{i=1}^l\sum_{j=1}^ka_{i,j}x^{(i)}_j$ and $\bar s\in\mathrm{Plm}_l([\mathbb{N}])^k$ we call the vector $\bar s(x) = \sum_{i=1}^l\sum_{j=1}^ka_{i,j}x^{(i)}_{s_i(j)}$ a plegma shift of the vector $x$.
\item[(iii)] The array $(x^{(i)}_j)_j$, $1\leq i\leq l$, will be called plegma spreading if for every $k\in\mathbb{N}$, every $\bar s\in\mathrm{Plm}_l[\mathbb{N}]^k$, and every $x = \sum_{i=1}^l\sum_{j=1}^ka_{i,j}x^{(i)}_j$ we have $\iii{x} = \iii{\bar s(x)}$.
\end{itemize}
\end{defn}

\begin{defn}
\label{definition joint spreading model}
Let $X$ be a Banach space, $l\in\mathbb{N}$, and $(x^{(i)}_j)_j$, $1\leq i\leq l$, be an array of sequences in $X$. Let also $E$ be a seminormed vector space and let $(e^{(i)}_j)_j$, $1\leq i\leq l$, be an array of sequences in $E$. We say that $(x^{(i)}_j)_j$, $1\leq i\leq l$, generates $(e^{(i)}_j)_j$, $1\leq i\leq l$, as a joint spreading model if for every $k\in\mathbb{N}$ and any vector $x = \sum_{i=1}^l\sum_{j=1}^ka_{i,j}x^{(i)}_j$ we have
\[\lim_{\substack{\min(s_1)\to\infty\\\bar s\in\mathrm{S}\text{-}\mathrm{Plm}_l([\mathbb{N}]^k)}}\|\bar s(x)\| = \iii{\sum_{i=1}^l\sum_{j=1}^ka_{i,j}e^{(i)}_j}.\]
Given a subset $A$ of $X$ we shall say that $A$ admits $(e^{(i)}_j)_j$, $1\leq i\leq l$, as a joint spreading model if there exists an array $(x^{(i)}_j)_j$, $1\leq i\leq l$, in $A$ that generates $(e^{(i)}_j)_j$, $1\leq i\leq l$ as a joint spreading model.
\end{defn}

The above notion was introduced in \cite{AGLM} and it was shown that every finite array $(x^{(i)}_j)_j$, $1\leq i\leq l$, of normalized Schauder basic sequences in a Banach space $X$ has a subarray $(x^{(i)}_{k_j})_j$ that generates some joint spreading model $(e_j^{(i)})_j$, $1\leq i\leq l$, which has to be a plegma spreading sequence.

Joint spreading models are a similar notion to that of asymptotic models,  from \cite{HO}, which was introduced and studied earlier. We modify the definition to make the connection to the above notions more clear.

\begin{defn}
Let $X$ be a Banach space, $(x^{(i)}_j)_j$, $i\in\mathbb{N}$ be an infinite array of normalized sequences in a Banach space $X$ and $(e_i)_i$ be a sequence in a seminormed space $E$. We say that $(x^{(i)}_j)_j$, $j\in\mathbb{N}$ generates $(e_i)_i$ as an asymptotic model if for any $l\in\mathbb{N}$ and vector $x = \sum_{i=1}^la_ix^{(i)}_1$ we have
\[ \lim_{\substack{\min(s_1)\to\infty\\ \bar s\in\mathrm{S}\text{-}\mathrm{Plm}_l([\mathbb{N}]^1)}}\left\|\bar s(x)\right\|  = \iii{\sum_{i=1}^la_ie_i}.
\]
\end{defn}

It was proved in \cite{HO} that any array $(x^{(i)}_j)_j$, $i\in\mathbb{N}$ of normalized sequences that are all weakly null have a subarray $(x^{(i)}_{j_k})_k$, $i\in\mathbb{N}$ that generates a 1-suppression unconditional asymptotic model $(e_i)_i$.

\subsection{Global asymptotic notions}

We first remind the definition of  an asymptotic-$\ell_p$ Banach space with a basis,  introduced by V. D. Milman and N. Tomczak-Jaegermann in \cite{MT}, and then we remind a coordinate free version of this definition from \cite{MMT}.

\begin{defn}
\label{asymptotic ellp}
Let $X$ be a Banach spaces with a Schauder basis $(e_i)_i$ and $1\leq p <\infty$. We say that the Schauder basis $(e_i)_i$ of $X$ is asymptotic-$\ell_p$ if there exist positive constants $D_1$ and $D_2$ so that for all $n\in\N$ there exists $N(n)\in\N$ so that whenever $N(n)\leq x_1 <\cdots <x_n$ are vectors in $X$ then
\begin{equation*}
 \frac{1}{D_1}\left(\sum_{i=1}^n\|x_i\|^p\right)^{1/p} \leq \left\|\sum_{i=1}^nx_i\right\| \leq D_2\left(\sum_{i=1}^n\|x_i\|^p\right)^{1/p}.
\end{equation*}
Specifically we say that $(e_i)_i$ is $D$-asymptotic-$\ell_p$ for $D = D_1D_2$. The definition of an asymptotic-$c_0$-space is given similarly.
\end{defn}

The classical examples of non-trivial asymptotic-$\ell_p$ spaces are Tsirelson's original Banach space from \cite{T} that is asymptotic-$c_0$ and the space constructed in \cite{FJ} (nowadays called Tsirelson space) that is asymptotic-$\ell_1$.

\begin{rem}
\label{asymptotic lp versions}
The definition above depends on the basis of $X$ and not only on $X$. A more general coordinate free version of the above for a whole space $X$ being asymptotic-$\ell_p$ can be found in \cite[Subsection 1.7]{MMT} (see also \cite{O}) and it is based on a game of two players. For each $n\in\N$ there is a version of this game that takes place in $n$ consecutive turns. In each turn $k$ of the game player (S) chooses a co-finite dimensional subspace $Y_k$ of $X$ and then player (V) chooses a normalized vector $y_k\in Y_k$. One of the formulations for being asymptotic-$\ell_p$ in this setting is that there exists $C$ so that for every $n\in\N$ player (S) has a wining strategy to force in $n$ turns player (V) to choose a sequence $(y_i)_{i=1}^n$ that is $C$-equivalent to the unit vector basis of $\ell_p^n$. Although this is not the initial formulation it is equivalent and this follows from \cite[Subsection 1.5]{MMT}.   Using this definition it is easy to show that if $X$ has a Schauder basis that is asymptotic-$\ell_p$ then $X$ is asymptotic-$\ell_p$. It also follows fairly easily that if a space $X$ is asymptotic-$\ell_p$ then it contains an asymptotic-$\ell_p$ sequence. In particular, a Banach space contains an asymptotic-$\ell_p$ subspace if and only if it contains an asymptotic-$\ell_p$ sequence.
\end{rem}

\subsection{Uniqueness of asymptotic notions}
\label{uniqueness sec}
The main purpose of this section is to discuss the property of a Banach space to exhibit a unique behavior with respect to the various asymptotic notions. Of particular interest to us is the question as to whether uniqueness with respect to one notion implies uniqueness with respect to another.

Throughout this subsection we let $\mathscr{F}$ denote one of two collections of normalized Schauder basic sequences in a given Banach space $X$, namely either $\mathscr{F}_0$, i.e., the collection of all normalized weakly null Schauder basic sequences, or $\mathscr{F}_b$, i.e. the collection of all normalized block sequences, if $X$ is assumed to have a basis.

\begin{defn}
Let $X$ be a Banach space and $\mathscr{F} = \mathscr{F}_0$ or $\mathscr{F} = \mathscr{F}_b$.
\begin{itemize}

\item[(i)] We say that $X$ admits a unique spreading model with respect to $\mathscr{F}$ if any two spreading models generated by sequences in $\mathscr{F}$ are equivalent.

\item[(ii)] We say that $X$ admits a uniformly unique spreading model with respect to $\mathscr{F}$ if there exists $C\geq 1$ so that any two spreading models generated by sequences in $\mathscr{F}$ are $C$-equivalent.

\end{itemize}
\end{defn}

The following is an open problem (see e.g. \cite[(Q8) on page 419]{O1}). 
\begin{problem}
Let $X$ be a Banach space and $\mathscr{F} = \mathscr{F}_0$ or $\mathscr{F} = \mathscr{F}_b$. Assume that $X$ admits a unique spreading model with respect to $\mathscr{F}$. Is this spreading model equivalent to the unit vector basis of some $\ell_p$, $1\leq p<\infty$, or $c_0$?
\end{problem}
It is well know that if the spreading model is uniformly unique then the answer is affirmative. This follows from an argument mentioned in \cite[Subsection 1.6.3]{MMT}.

\begin{defn}
Let $X$ be a Banach space and $\mathscr{F} = \mathscr{F}_0$ or $\mathscr{F} = \mathscr{F}_b$. We say that $X$ admits a unique joint spreading model with respect to $\mathscr{F}$ if there exists a constant $C$ so that for any $l\in\mathbb{N}$ and any two $l$-arrays generated as joint spreading models by $l$-arrays in $\mathscr{F}$ are $C$-equivalent.
\end{defn}

The existence of a uniform constant is included in the definition of unique joint spreading models. The reason for this is to separate uniqueness of spreading models from uniqueness of joint spreading models. If one assumes that $X$ admits a unique spreading model with respect to $\mathscr{F}$ then it follows that all $l$-joint spreading models generated by weakly null $l$-arrays in $\mathscr{F}$ are equivalent as well.

We remind that it was proved in \cite{AGLM} that if a Banach space $X$ admits a unique joint spreading model with respect to $\mathscr{F}$ then $X$ satisfies a property called the uniform approximation on large subspace. This is a property of families of bounded linear operators on $X$.

\begin{defn}
Let $X$ be a Banach space and $\mathscr{F} = \mathscr{F}_0$ or $\mathscr{F} = \mathscr{F}_b$. We say that $X$ admits a unique asymptotic model with respect to  $\mathscr{F}$ if any two asymptotic models generated by arrays of sequences in $\mathscr{F}$ are equivalent.
\end{defn}
It can be seen that if $X$ has a unique asymptotic model with respect to $\mathscr{F}$ then there must exist a $C$ so that any two asymptotic models generated by arrays of sequences in $\mathscr{F}$ are $C$-equivalent. This is because asymptotic models are generated by infinite arrays.

As it was mentioned in passing in \cite{AGLM} uniqueness of joint spreading models and uniqueness of asymptotic models are equivalent. We briefly describe a proof.
\begin{prop}
Let $X$ be a Banach space and $\mathscr{F} = \mathscr{F}_0$ or $\mathscr{F} = \mathscr{F}_b$. Then $X$ admits a unique joint spreading model with respect to $\mathscr{F}$ if and only if it admits a unique asymptotic model with respect to $\mathscr{F}$.
\end{prop}
\begin{proof}
 If $X$ admits a unique asymptotic model then, as it was mentioned above, it does so for a uniform constant $C$. We start with two $l$-arrays $(x_j^{(i)})_j$, $(y^{(i)}_j)_j$, $1\leq i\leq l$, generating joint spreading models $(e_j^{(i)})_j$, $(d_j^{(i)})$, $1\leq i\leq l$, which we will show that they are equivalent. Define the infinite arrays $(\tilde x_j^{(i)})$, $(\tilde y_j^{(i)})$, $i\in\mathbb{N}$ given by $\tilde x^{(ml + i)}_j = x_j^{(i)}$ and $\tilde y^{(ml + i)}_j = y_j^{(i)}$ for $m\in\mathbb{N}\cup\{0\}$, $1\leq i\leq l$, and $j\in\mathbb{N}$. Any asymptotic model $(e_i)_i$ generated by a subarray of $(\tilde x_j^{(i)})_j$, $i\in\mathbb{N}$, is isometrically equivalent to $(e_j^{(i)})_j$, $1\leq i\leq l$ by mapping $e_{ml+i}$ to $e^{(i)}_{l}$, for $m\in\mathbb{N}\cup\{0\}$, $1\leq i\leq l$. We can make a similar observation about any asymptotic model $(d_i)_i$ generated by a subarray of $(\tilde y_j^{(i)})$, $i\in\mathbb{N}$. As $(e_i)_i$ and $(d_i)_i$ are $C$-equivalent we deduce that the same is true for $(e_j^{(i)})_j$, $(d_j^{(i)})$, $1\leq i\leq l$. The inverse implication is slightly easier. If we assume that there is $C$ so that for any $l\in\mathbb{N}$ any $l$-joint spreading models admitted by $l$-arrays in $\mathscr{F}$ then it is almost straightforward that the first $l$ elements of any two asymptotic models generated by arrays in $\mathscr{F}$ are $C$-equivalent.
\end{proof}

If a space admits a unique asymptotic model, and hence also spreading model, then it has to be equivalent to the unit vector basis of $\ell_p$ or $c_0$.  This follows, e.g., from the uniform uniqueness of the spreading model.

We now compare  uniqueness of the various asymptotic notions. Here, $1\leq p\leq \infty$ and whenever $p=\infty$ then $\ell_p$ should be replaced with $c_0$. The implications presented in the next statement are fairly obvious. 

\begin{prop}
Let $1\leq p\leq \infty$, $X$ be a Banach space, and $\mathscr{F} = \mathscr{F}_0$ or $\mathscr{F} = \mathscr{F}_b$. Consider the following properties.
\begin{itemize}
\item[(a1)$_p$] The space $X$ is coordinate free asymptotic-$\ell_p$.
\item[(a2)$_p$] The space $X$ has a basis that is asymptotic-$\ell_p$.
\item[(b)$_p$] The space $X$ admits a unique $\ell_p$ asymptotic model with respect to $\mathscr{F}$.
\item[(c)$_p$] The space $X$ admits a uniformly unique $\ell_p$ spreading model with respect to $\mathscr{F}$.
\item[(d)$_p$] The space $X$ admits a unique $\ell_p$-spreading model with respect to $\mathscr{F}$.
\end{itemize}
Then (a1)$_p \vee$(a2)$_p \Rightarrow$(b)$_p \Rightarrow$(c)$_p \Rightarrow$(d).
\end{prop}

The question as to whether any inverse implications hold is somewhat less straightforward. We can divide this problem into questions of separation and complete separation (see Definition on page \pageref{property separation}). We discuss this topic starting with the bottom of the list and moving upwards.

\begin{question}
Let $X$ be a Banach space and $\mathscr{F} = \mathscr{F}_0$ or $\mathscr{F} = \mathscr{F}_b$. If $X$ admits a unique spreading model with respect to $\mathscr{F}$ does it also admit a uniformly unique spreading model with respect to $\mathscr{F}$?
\end{question}

In other words, can property (c) be separated from (d). This can be answered fairly easily. Fix $1<p<\infty$ and consider for each $n\in\mathbb{N}$ a norm on $\ell_p$ given by $\|x\|_n = \|x\|_\infty\vee \|x\|_{\ell_p}$. The space $X = (\sum_n\oplus (\ell_p,\|\cdot\|_n))_{\ell_p}$ admits a unique $\ell_p$-spreading model with respect to $\mathscr{F}_0$ but not a uniformly unique $\ell_p$-spreading model with respect to $\mathscr{F}_0$. A slightly less trivial example can be given for $p=1$ by using e.g. a norm $\|x\|_n$ defined on $T$ and taking a $T$-sum. Interestingly it is not possible to do this for $c_0$. It follows from \cite[Proposition 3.2]{AOST} that if a space $X$ admits a unique $c_0$ spreading model with respect to $\mathscr{F}_0$ then this has to happen uniformly. The, more interesting, complete separation analogue of the above question is the following.

\begin{question}
Let $X$ be a Banach space and $\mathscr{F} = \mathscr{F}_0$ or $\mathscr{F} = \mathscr{F}_b$. If $X$ admits a unique spreading model with respect to $\mathscr{F}$ does $X$ have a subspace $Y$ that admit a uniformly unique spreading model with respect to $\mathscr{F}$?
\end{question}

This is less obvious. For example, if one considers $X = (\sum_n\oplus (\ell_2,\|\cdot\|_n))_{\ell_2}$ then $\ell_2$ is a subspace of $X$. To answer this question, in Section \ref{the other space} we construct a Banach space $\Xt$ with a unique $\ell_1$ spreading model with respect to $\mathscr{F}_0$ so that in every subspace of $\Xt$ one can find normalized weakly null sequences generating a spreading model with an arbitrarily ``bad'' equivalence to the unit vector basis of $\ell_1$.

\begin{question}
Let $X$ be a Banach space and $\mathscr{F} = \mathscr{F}_0$ or $\mathscr{F} = \mathscr{F}_b$. If $X$ admits a uniformly unique spreading model with respect to $\mathscr{F}$ does $X$ admit a uniformly unique asymptotic model with respect to $\mathscr{F}$?
\end{question}

The answer to the above question is negative in all cases of unique spreading models (which have to be some $\ell_p$, $1\leq p<\infty$, or $c_0$). It was observed in \cite{BLMS} that the space $T^*(T^*)$ admits $c_0$ as a uniformly unique spreading model whereas the space admits the unit vector basis of $T^*$ as an asymptotic model. Proposition 3.12 of \cite{BLMS} can also be used to show that $T(T)$ admits a uniformly unique $\ell_1$-spreading model, yet $T(T)$ admits the unit vector basis of $T$ as an asymptotic model. We can replace $T$ with $T_p$, the $p$-convexification of $T$, for $1<p<\infty$. It follows, again from , \cite[Proposition 3.12]{BLMS} that $T_p(T_p)$ has a uniformly unique $\ell_p$ spreading model. Is also easy to see that $T_p(T_p)$ admits the unit vector basis of $T_p$ as an asymptotic model.

\begin{question}
Let $X$ be a Banach space and $\mathscr{F} = \mathscr{F}_0$ or $\mathscr{F} = \mathscr{F}_b$. If $X$ admits a uniformly unique spreading model with respect to $\mathscr{F}$ does $X$ have a subspace that admits a uniformly unique asymptotic model with respect to $\mathscr{F}$?
\end{question}

We prove in this paper that the answer to the above question is conclusively negative, regardless of the assumption on the unique spreading model. We construct a Banach space $\X$ that admits a uniformly unique $\ell_1$-spreading mode so that every block subspace of $\X$ admits a $c_0$ asymptotic model. We also prove that $\X^*$ admits a uniformly unique $c_0$-spreading model and that every block subspace of $\X$ admits an $\ell_1$ asymptotic model. We also describe, for $1<p<\infty$, the construction of a space $\X^p$ that admits a uniformly unique $\ell_p$-spreading mode so that every block subspace of $\X$ admits a $c_0$ asymptotic model.

We remind that according to Remark \ref{asymptotic lp versions} a Banach space contains an asymptotic-$\ell_p$ subspace with a basis if and only if it contains a coordinate free asymptotic-$\ell_p$ subspace.
\begin{question}[E. Odell (Q7) \cite{O1} \& page 66 \cite{O} and M. Junge, D. Kutzarova, E. Odell Problem 1.2 \cite{JKO}]
Let $X$ be a Banach space that admits a uniformly unique spreading model with respect to $\mathscr{F}$. Does $X$ have an asymptotic-$\ell_p$ or asymptotic-$c_0$ subspace?
\end{question}
The spaces $\X$, $\X^*$, and $\X^p$, $1<p<\infty$, provide a negative answer to the above question for all possible assumptions on the unique spreading model.

\begin{question}
Let $X$ be a Banach space and $\mathscr{F} = \mathscr{F}_0$ or $\mathscr{F} = \mathscr{F}_b$. If $X$ admits a unique asymptotic model with respect to $\mathscr{F}$ is $X$ asymptotic-$\ell_p$ or asymptotic-$c_0$ in the coordinate free sense of \cite{MMT}?
\end{question}
Interestingly, for this question the type of unique spreading model makes a difference to the result. It was proved in \cite{FOSZ} that if a separable Banach space $X$ contains no copy of $\ell_1$ and $X$ has a unique $c_0$ asymptotic model with respect to $\mathscr{F}_0$ then $X$ is asymptotic-$c_0$ (in the sense of \cite{MMT}). Replacing $c_0$ with $\ell_p$, for $1<p<\infty$, completely changes the situation. In \cite[Subsection 7.2]{BLMS}, for each $1<p<\infty$ a reflexive Banach space is presented all asymptotic models of which are isometrically equivalent to the unit vector basis of $\ell_p$, yet the space is not asymptotically-$\ell_p$, in the sense of \cite{MMT}. A slightly different approach to the same question is based on a construction in \cite[Example 4.2]{OS3}. One can consider an infinite hight countably branching and well founded tree $\mathcal{T}$. Then, for $1<p<\infty$, define a norm on $c_{00}(\mathcal{T})$ as follows. If $x = \sum_{\lambda\in\mathcal{T}}c_\lambda e_\lambda$ then set
\[\|x\| = \sup\left\{\left(\sum_{i=1}^m\left(\sum_{\lambda \in \beta_i}|c_\lambda|\right)^p\right)^{1/p}: (\beta_i)_{i=1}^m\text{ are disjoint segments of }\mathcal{T}\right\}.\]
One can show, using \cite[Proposition 3.12]{BLMS} and induction on the hight of $\mathcal{T}$, that the completion of this space has only the unit vector basis of $\ell_p$ as an asymptotic model and it is not asymptotically-$\ell_p$.

The Definition from \cite[Subsection 7.2]{BLMS} also yields a non-reflexive Banach space with an unconditional Schauder basis that admits the unit vector basis of $\ell_1$ as a unique asymptotic model with respect to all arrays of block sequences of the basis yet the space is not asymptotic-$\ell_1$. In fact, this space is a Schur space. The first example of a reflexive non-asymptotic-$\ell_1$ space with a unique $\ell_1$ asymptotic model was first given in \cite{AGM}.

The following open question is the remaining implication from the list and it first appeared in \cite[Problem 6.1]{HO}.

\begin{problem}
Let $1\leq p<\infty$ and $X$ be a Banach space not containing $\ell_1$ so that every asymptotic model generated by a weakly null array in $X$ is equivalent to the unit vector basis of $\ell_p$. Does $X$ contain an asymptotic-$\ell_p$-subspace?
\end{problem}

\subsection{Finite block representability}
In this part of this section we recall the notion of finite block representability and the Krivine set of a space.

\begin{defn}
Let $X$ be a Banach space with a Schauder basis $(e_i)_i$ and let also $Y$ be a finite dimensional Banach space with a Schauder basis  $(y_i)_{i=1}^n$. We say that $(y_i)_{i=1}^n$ is block representable in $X$ if for every $\e>0$ there exists a block sequence $(x_i)_{i=1}^n$ in $X$ that is $(1+\e)$-equivalent to $(y_i)_{i=1}^n$. Given an infinite dimensional Banach space $Z$ with a Schauder basis $(z_i)_i$ we say that $(z_i)_i$ is finitely block representable in $X$ if for every $n\in\N$ the sequence $(z_i)_{i=1}^n$ is block representable in $X$.
\end{defn}

Given a Banach space $X$ with a basis the Krivine set $K(X)$ of $X$ is the set of all $p\in [1,\infty]$ so that the unit vector basis of $\ell_p$ (or of $c_0$ in the case $p=\infty$) is finitely block representable in $X$. It was proved by J-L Krivine in \cite{K} that this set is always non-empty. It is observed in \cite[Subsection 1.6.3]{MMT} that a stronger result holds, namely that there is $p\in[1,\infty]$ so that for all $\e>0$ and $n\in\N$ there exists a block sequence $(x_i)_{i=1}^\infty$ so that for all $k_1<\cdots<k_n$ the sequence $(x_{k_i})_{i=1}^n$ is $(1+\e)$-equivalent to the unit vector basis of $\ell_p^n$. We shall refer to the set of all such $p$'s as the strong Krivine set of $X$ and denote it by $\widetilde K(X)$. Clearly, $\widetilde K(X)\subset K(X)$.
It is clear that if $X$ is asymptotic-$\ell_p$, for some $1\leq p \leq\infty$ then $K(X) = \widetilde K(X) = \{p\}$.

\begin{question}
Let $X$ be a Banach space with a basis. Does there exist a block subspace $Y$ of $X$ so that $K(Y) = \tilde K(Y)$?
\end{question}
We answer with question negatively by showing that for every block subspace $Y$ of $\X$ we have $\tilde K(Y) = \{1\} \subsetneq [1,\infty] = K(Y)$. We also point out that for every $1<p<\infty$ and every block subspace $Y$ of $\X^p$ we have $\tilde K(Y) = \{p\} \subsetneq [p,\infty] = K(Y)$.

We additionally show that all 1-unconditional sequences are finitely block representable in every block subspace $Y$ of $\X$. To show this we use a result from \cite{OS1} where it was observed that there is a family of finite unconditional sequences that is universal for all unconditional sequences.
\begin{prop}[\cite{OS1}]
\label{universal for unc}
Let $n\in\N$ and $X_n$ be the finite dimensional space spanned by the sequence $(e_{i,j})_{i,j=1}^n$ ordered lexicographically and endowed with the norm
\begin{equation*}
 \left\|\sum_{i=1}^n\sum_{j=1}^na_{i,j}e_{i,j}\right\| = \max_{1\leq j\leq n}\sum_{i=1}^n|a_{i,j}|.
\end{equation*}
If $X$ is a Banach space with a Schauder basis $(x_i)_i$ so that for each $n\in\N$ the sequence $(e_{i,j})_{i,j=1}^n$ is block representable in $X$, then every 1-unconditional basic sequence is finitely block representable in $X$.
\end{prop}

\subsection{Asymptotically symmetric spaces}
This is the final part of this section and we remind the notion of an asymptotically symmetric Banach space. It was introduced in \cite{JKO} and the motivation stems from the theory of non-commutative $L_p$ spaces.
\begin{defn}
\label{defas}
A Banach space $X$ is called asymptotically symmetric if there exists $C>0$ so that for all $l\in\N$, all  bounded arrays of sequences $(x_j^{(i)})_j$, $1\leq i\leq l$ in $X$, and all permutations $\sigma$ of $\{1,\ldots,l\}$ we have
\begin{equation}
\label{a.s.}
\lim_{j_1\to\infty}\cdots\lim_{j_l\to\infty}\Big\|\sum_{i=1}^lx^{(i)}_{j_i}\Big\| \leq C\lim_{j_{\sigma(1)}\to\infty}\cdots\lim_{j_{\sigma(l)}\to\infty}\Big\|\sum_{i=1}^lx^{(i)}_{j_i}\Big\|
\end{equation}
provided that both iterated limits exist.
\end{defn}
This is a notion that is weaker than the one of stable Banach spaces. It also follows from the discussion leading up to \cite[Proposition 1.1]{JKO} that a reflexive asymptotic-$\ell_p$ space is asymptotically symmetric. It was also observed there that $L_p$ provides a counterexample to the converse.

\begin{question}[Junge, D. Kutzarova, E. Odell Problem 0.2 \cite{JKO}]
Let $X$ be an asymptotically symmetric Banach space. Does $X$ contain an asymptotic-$\ell_p$ (or asymptotic-$c_0$) subspace?
\end{question}
It turns out that the spaces $\X$ and $\X^*$ are asymptotically symmetric and therefore each of them provides a negative answer to the above question. This is an immediate consequence of the next result, which follows easily from \cite{JKO}. We include a proof for completeness.

\begin{prop}
\label{unique 1 or 0 am}
Let $X$ be a reflexive Banach space that satisfies one of the following conditions.
\begin{itemize}
\item[(i)] The space $X$ has a Schauder basis $(e_i)_i$ and it admits a uniformly unique $\ell_1$ spreading model with respect to $\mathscr{F}_b$.
\item[(ii)] The space $X$ is separable and it admits a unique $c_0$ spreading model with respect to $\mathscr{F}_0$.
\end{itemize}
Then $X$ is asymptotically symmetric.
\end{prop}

\begin{proof}
The statement of \cite[Theorem 2.3]{JKO} is that if a (not necessarily reflexive) Banach space satisfies (i) then it is block asymptotically symmetric, i.e., it satisfies \eqref{a.s.} for arrays of bounded block sequences in $X$. The statement of \cite[Theorem 1.1 (c)]{JKO} is that when $X$ is reflexive with a basis then being asymptotic symmetric is equivalent to being block asymptotic symmetric. Similarly, \cite[Theorem 2.4]{JKO} yields that any Banach space satisfying (ii) is weakly asymptotically symmetric and once more \cite[Theorem 1.1 (c)]{JKO} states that for reflexive spaces this is equivalent to being asymptotically symmetric.
\end{proof}

\section{Definition of the space $\X$}\label{definition section}
We define the space $\X$ by first defining a norming set $\W$. This is a norming set of the mixed-Tsirelson type with certain constraints applied to the weights of the functionals used in the construction. Fix a pair of strictly increasing sequences of natural numbers $(m_j)_j$, $(n_j)_j$ with $m_1 = 2$ and $n_1 = 1$ satisfying the growth conditions
\begin{itemize}
\item[(i)] for all $C>1$ we have $\displaystyle{\lim_j\frac{C^{n_j}}{m_j} = \infty}$,

\item[(ii)] $\displaystyle{\lim_j\frac{m_j}{m_{j+1}} = 0}$, and

\item[(iii)] $n_{j+1} > n_{j_1}+\cdots+n_{j_l} + 1$ for all $l\in\N$ and $1\leq j_1,\ldots,j_l \leq j$ with the property $m_{j_1}\cdots m_{j_l} < m_{j+1}^2$.   
\end{itemize}
These properties can be achieved by taking any strictly increasing sequence of natural numbers $(m_j)_j$, with $m_1 = 2$, satisfying (ii) and afterwards choosing any strictly increasing sequence of natural numbers $(n_j)_j$, satisfying $n_1 = 1$ and so that $n_{j+1} > n_j\log(m^2_{j+1})$ for all $j\in\N$.
\begin{ntt}
Let $G$ be a subset of $c_{00}(\N)$.
\begin{itemize}
\item[(i)] Given $j_1,\ldots,j_l\in\N$ and $\mathcal{S}_{n_{j_1}+\cdots+n_{j_l}}$ admissible functionals $f_1 < \cdots < f_d$ in $G$ we call a functional of the form
\begin{equation*}
f = \frac{1}{m_{j_1}\cdots m_{j_l}}\sum_{q = 1}^d f_q
\end{equation*}
a weighted functional of $G$ of weight $w(f) = m_{j_1}\cdots m_{j_l}$ and vector weight $\vec w(f) = (j_1,\ldots,j_l)$. For all $i\in\N$, we also call $f = \pm e_i^*$ a weighted functional of weight $w(f) = \infty$ and in this case we do not define $\vec w(f)$.

\item[(ii)] A (finite or infinite) sequence $f_1<f_2<\cdots<f_q<\cdots$ of weighted functionals of $G$ is called very fast growing if $w(f_q) > \max\supp(f_{q-1})$ for $q>1$.
\end{itemize}
\end{ntt}
Note that if $(f_q)_q$ is a sequence of very fast growing weighted functionals then any of the $f_q$'s may be of the form $\pm e_i^*$ for $i\in\N$. Furthermore, the weight and vector weight of a functional may not be uniquely defined but this causes no problems.

\begin{defn}
Let $\W$ be the smallest subset of $c_{00}(\N)$ that satisfies the following two conditions.
\begin{itemize}

\item[(i)] $\pm e_i^*$ is in $\W$ for all $i\in\N$ and

\item[(ii)] for every $j_1,\ldots,j_l\in\N$, and every $\mathcal{S}_{n_{j_1}+\cdots+n_{j_l}}$-admissible and very fast growing sequence of weighted functionals $(f_q)_{q=1}^d$ in $\W$ the functional $$f = \frac{1}{m_{j_1}\cdots m_{j_l}}\sum_{q=1}^df_q$$ is in $\W$.

\end{itemize}
We define a norm on $c_{00}(\N)$ given by $\|x\| = \sup\{f(x): x\in \W\}$ and we set $\X$ to be the completion of $(c_{00}(\N), \|\cdot\|)$.
\end{defn}

\begin{rem}
\label{inductive construction norming}
Alternatively the set $\W$ can be defined to be the increasing union of a sequence of sets $(W_n)_{n=0}^\infty$ where $W_0 = \{\pm e_i: i\in\N\}$ and
 \begin{equation*}
 \begin{split}
W_{n+1} = W_n\cup\Bigg\{& \frac{1}{m_{j_1}\cdots m_{j_l}} \sum_{q=1}^d f_q: ~ j_1,\ldots,j_l\in\N, \mbox{ and }(f_q)_{q=1}^d \mbox{ is an}\\&
\mbox{$\mathcal{S}_{n_{j_1}+\cdots+n_{j_l}}$ admissible and very fast growing sequence}\\
& \mbox{of weighted functionals in $W_n$}\Bigg\}.
 \end{split}
 \end{equation*}
\end{rem}

\begin{rem}
By induction on $n$ it easily follows that each set $W_n$ is closed under changing signs and under taking projections onto subsets, hence the same holds for $\W$. This yields that the unit vector basis of $c_{00}(\N)$ forms a 1-unconditional basis for the space $\X$.
\end{rem}

\begin{rem}
It is easy to check by induction on the construction of $\W$ that for every $f\in\W$ each of it coordinates is either zero or of the form $1/d$ for  some non-zero integer $d$. As $\W$ is closed under projections onto arbitrary subsets, we deduce that for every $k\in\N$ the set $\W|_k$ of all $f\in\W$ with $\max\supp(f)\leq k$ is compact in the topology of point-wise convergence. This yields that for every $x\in\X$ with $\supp(x)$ finite there is $f\in\W$ with $f(x) = \|x\|$.  
\end{rem}

\section{The spreading model of block sequences in $\X$}
\label{uniform unique spreading model section}
We prove that every normalized block sequence in $\X$ has a subsequence that generates a $4$-$\ell_1$ spreading model. This is unusual for constructions using saturations under constraints where typically at least two different spreading models appear (see, e.g., \cite{AM1}). As it will be shown later the constraints impose a variety of asymptotic models and local block structure in $\X$.
\begin{prop}
\label{uniform ell1}
Let $(x_i)_i$ be a normalized block sequence in $\X$. Then there exists $L\in[\N]^\infty$ so that for every $j_0\in\N$, every $F\subset L$ with $(x_i)_{i\in F}$ being $\mathcal{S}_{n_{j_0}}$-admissible, and every scalars $(c_i)_{i\in F}$ we have
\begin{equation*}
\left\|\sum_{i\in F}c_ix_i\right\| \geq \frac{1}{2m_{j_0}}\sum_{i\in F}|c_i|. 
\end{equation*}
In particular, every normalized block sequence in $\X$ has a subsequence that generates a spreading model that is 4-equivalent to the unit vector basis of $\ell_1$.
\end{prop}

\begin{proof}
We quickly observe that the second statement quickly follows from the first one and $m_1 = 2$, $n_1 = 1$. We now proceed to prove the first statement. For every $k\in\N$ choose $f_k\in\W$ with $f(x_k) = 1$ so that $\ran(f_k)\subset \ran(x_k)$. We distinguish two cases, namely the one in which $\limsup_kw(f_k)$ is finite and the one in which it is infinite.

In the first case, take an infinite subset $L$ of $\N$ and $j_1,\ldots,j_l\in\N$ so that for all $k\in\N$ we have $\vec w(f_k) = m_{j_1}\cdots m_{j_l}$. For each $k\in L$ write
\begin{equation*}
f_k = \frac{1}{m_{j_1}\cdots m_{j_l}}\sum_{q=1}^{d_k}f_q^k 
\end{equation*}
where each sequence $(f_q^k)_{q=1}^{d_k}$ is $\mathcal{S}_{n_{j_1}+\cdots+n_{j_l}}$ admissible and very fast growing with $\min\supp(x_k) \leq \max\supp(f_1^l) < w(f_2^k)$, which implies that the sequence $((f_q^k)_{q=2}^{d_k})_{k\in L}$, enumerated in the natural way, is very fast growing. Also, for every $k_1<\cdots<k_d$ in $L$ so that $(x_k)_{k\in F}$ is $\mathcal{S}_{n_{j_0}}$-admissible the functionals $((f_q^{k_i})_{q=2}^{d_k})_{i=1}^n$ are $\mathcal{S}_{n_{j_1}+\cdots+n_{j_l} + n_{j_0}}$ admissible and it follows that the functional
\begin{equation*}
\label{uniform ell1 eq1}
f = \frac{1}{m_{j_1}\cdots m_{j_l}m_{j_0}}\sum_{i=1}^n\sum_{q=2}^{d_{k_i}}f_q^{k_i} \text{ is in }\W.
\end{equation*}

As for each $k\in\N$ the functional $f_1^k\in\W$ we have $f_1^k(x_k) \leq 1$ and therefore
\begin{equation}
\label{uniform ell1 eq 1}
\frac{1}{m_{j_1}\cdots m_{j_l}}\sum_{q=2}^{d_k}f_q^k(x_k) \geq f(x_k) - \frac{1}{m_{j_1}\cdots m_{j_l}}f_1^k(x_k) \geq 1 - 1/2 =  1/2.
\end{equation}
For any $k_1<\cdots<k_n$ in $L$ so that $(x_k)_{k\in F}$ is $\mathcal{S}_{n_{j_0}}$-admissible and scalars $a_1,\ldots,a_n$ we conclude
\begin{equation*}
\begin{split}
 \left\|\sum_{i=1}^na_ix_{k_i}\right\| & = \left\|\sum_{i=1}^n|a_i|x_{k_i}\right\| \geq \frac{1}{m_{j_1}\cdots m_{j_l}m_{j_0}}\sum_{i=1}^n\sum_{q=2}^{d_{k_i}}f_q^{k_i}\left(\sum_{j=1}^n|a_j|x_{k_j}\right)\\
 &= \frac{1}{m_{j_0}}\sum_{i=1}^n|a_i|\frac{1}{m_{j_1}\cdots m_{j_l}}\sum_{q=2}^{d_{k_i}}f_q^{k_i}\left(x_{k_i}\right)\\
 & \geq \frac{1}{m_{j_0}}\sum_{i=1}^n\frac{1}{2}|a_i| = \frac{1}{2m_{j_0}}\sum_{i=1}^n|a_i|.
\end{split}
\end{equation*}

In the second case we may choose an infinite subset of $L$ so that $(f_k)_{k\in L}$ is very fast growing. As $m_1 = 2$ and $n_1 = 1$ we deduce that for any $k_1<\cdots<k_n$ in $L$ so that $(x_k)_{k\in F}$ is $\mathcal{S}_{n_{j_0}}$-admissible the functional
\begin{equation*}
 f = \frac{1}{m_{j_0}}\sum_{i=1}^nf_{k_i}
\end{equation*}
is in $\W$. As before, for every $k_1<\cdots<k_n$ in $L$ so that $(x_k)_{k\in F}$ is $\mathcal{S}_{n_{j_0}}$-admissible and scalars $a_1,\ldots,a_n$ we conclude that $\|\sum_{i=1}^na_ix_{k_i}\| \geq (1/m_{j_0})\sum_{i=1}^n|a_i|$.
\end{proof}

An easy consequence of the above result is the following.
\begin{cor}
\label{strong Krivine singleton}
The strong Krivine set of $\X$ is $\widetilde K(X) = \{1\}$. 
\end{cor}

\section{The auxiliary space}\label{section aux}

For every $N$ we define an auxiliary space that is defined by a norming set $W^N_\mathrm{aux}$ very similar to $\W$. The reason for which we define an infinite family of auxiliary spaces is because we are interested in the almost isometric representation of finite unconditional sequences as block sequences in $\X$. To define this norming set we slightly alter the notions of weighted functionals and very fast growing sequences. In this case, given a subset $G$ of $c_{00}(\N)$ we will call a functional $f$ an auxiliary weighted functional of weight $w(f) = m_{j_1}\cdots m_{j_l}$ and vector weight $\vec w(f) = (m_{j_1},\ldots,m_{j_l})$, for $j_1,\ldots,j_n\in\N$, if it is of the form
\begin{equation*}
 f = \frac{1}{m_{j_1}\cdots m_{j_l}}\sum_{q=1}^df_q
\end{equation*}
where the functionals $(f_q)_{q=1}^d$ are in $G$ and they are $\mathcal{S}_{n_{j_1}+\cdots+n_{j_l}}\ast\mathcal{A}_3$ admissible. For all $i\in\N$ we will also say that $f = \pm e_i^*$ is an auxiliary weighted functional of weight $w(f) = \infty$ and we do not define $\vec w(f)$ in this case. A sequence of auxiliary weighted functionals $(f_q)_q$ will be called $N$-sufficiently large if $w(f_q) > N$ for $q\geq2$. There is no restriction on $w(f_1)$.


\begin{defn}
\label{aux}
For $N\in\N$ let $W_\mathrm{aux}^N$ be the smallest subset of $c_{00}(\N)$ that satisfies the following to conditions.
\begin{itemize}

\item[(i)] $\pm e_i^*$ is in $W^N_\mathrm{aux}$ for all $i\in\N$ and

\item[(ii)] for every $j_1,\ldots,j_l\in\N$ and every $\mathcal{S}_{n_{j_1}+\cdots+n_{j_l}}\ast\mathcal{A}_3$ admissible sequence of $N$-sufficiently large auxiliary weighted functionals $(f_q)_{q=1}^d$ in $W^N_\mathrm{aux}$ the functional $$f = \frac{1}{m_{j_1}\cdots m_{j_l}}\sum_{q=1}^df_q$$ is in $W^N_\mathrm{aux}$.

\end{itemize}
We define a norm $\|\cdot\|_{\mathrm{aux},N}$ on $c_{00}(\N)$ by setting  $\|x\|_{\mathrm{aux},N} = \sup\{f(x): f\in W_\mathrm{aux}^N\}$ for $x\in c_{00}(\N)$.
\end{defn}


\begin{rem}
 As in Remark \ref{inductive construction norming} the set $W_\mathrm{aux}^N$ can be defined as an increasing union of sets $(W_n^N)_{n=0}^\infty$ where $W^N_0 = \{\pm e_i:i\in\N\}$ and for each $n\in\N$ the set $W_{n+1}^N$ is defined by using $N$-sufficiently large $\mathcal{S}_{n_{j_1}+\cdots+n_{j_l}}\ast\mathcal{A}_3$ admissible sequences in $W_n^N$.
\end{rem}

The purpose of the following two lemmas is to bound the norm of linear combinations of certain vectors in the auxiliary spaces from above. The final estimate of this section is \eqref{upper auxiliary eq} which will be used to bound the norm of appropriately chosen vectors in $\X$.

\begin{lem}
\label{basis on basis}
Let $j_0\in\N$, $\e>0$, $x = \sum_{r\in F}c_re_r$ be a $(n_{j_0}-1,\e)$ basic s.c.c., and $\tilde x = m_{j_0}x$. Let also $j_1,\ldots,j_l\in\N$ with $\max_{1\leq i\leq l} j_i \neq j_0$, $G\in \mathcal{S}_{n_{j_1}+\cdots+n_{j_l}}\ast\mathcal{A}_3$ and $f = (m_{j_1}\cdots m_{j_l})\sum_{i\in G}e_i^*$. Then
\begin{equation}
\label{basis on basis eq}
|f(x)| \leq  \max\left\{2\e m_{j_0},\frac{m_{j_0}}{m_{j_0+1}},\frac{1}{m_{j_0}}\right\}.
\end{equation}
\end{lem}

\begin{proof}
If $\max_{1\leq i\leq l} j_i > j_0$ then $\|f\|_\infty \leq 1/(m_{j_0 +1})$ which yields
\begin{equation}
\label{basis on basis eq proof 1}
|f(\tilde x)| \leq \|f\|_\infty\|\tilde x\|_1 \leq \frac{m_{j_0}}{m_{j_0+1}}.
\end{equation}
If $\max_{1\leq i\leq l} j_i < j_0$ we distinguish two cases, namely whether $n_{j_1}+\cdots+n_{j_l} < n_{j_0} - 1$ or otherwise. In the first case, as $G\in\mathcal{S}_{n_{j_1}+\cdots+n_{j_l}}\ast\mathcal{A}_3$ we obtain
\begin{equation}
\label{basis on basis eq proof 2}
|f(\tilde x)| \leq \frac{m_{j_0}}{m_{j_1}\cdots m_{j_l}}\sum_{i\in G\cap F}c_i \leq \frac{m_{j_0}}{2}3\e.
\end{equation}
If on the other hand $\max_{1\leq i\leq l} j_i < j_0$ and $n_{j_1}+\cdots+n_{j_l}\geq n_{j_0} - 1$, by property (iii) of the sequences $(m_j)_j$, $(n_j)_j$ we obtain $m_{j_1}\cdots m_{j_l} \geq m_{j_0}^2$ which gives $\|f\|_\infty \leq 1/m_{j_0}^2$. We conclude
\begin{equation}
\label{basis on basis eq proof 3}
|f(\tilde x)| \leq \|f\|_\infty\|\tilde x\|_1 \leq \frac{m_{j_0}}{m_{j_0}^2} = \frac{1}{m_{j_0}}.
\end{equation}
The result follows from combining \eqref{basis on basis eq proof 1}, \eqref{basis on basis eq proof 2}, and \eqref{basis on basis eq proof 3}.
\end{proof}

\begin{lem}
\label{upper auxiliary}
Let $N,k,l\in\N$, $\e>0$, $(t_i)_{i=1}^k$ be pairwise different natural numbers and $(x_{i,j})_{1\leq i \leq k, 1\leq j\leq l}$ be vectors in $c_{00}(\N)$ so that for each $i,j$ the vector $x_{i,j}$ is of the form
\begin{equation}
x_{i,j} = m_{t_i}\tilde x_{i,j},\text{ where } \tilde x_{i,j} = \sum_{r\in F_{i,j}}c_r^{i,j}e_r \text{ is a } (n_{t_i}-1,\e) \text{ basic s.c.c.}
\end{equation}
Then, for any scalars $(a_{i,j})_{1\leq i \leq k, 1\leq j\leq l}$ and $f\in W_\mathrm{aux}^N$ we have
\begin{equation}
\label{upper auxiliary eq}
\left|f\left(\sum_{j=1}^l\sum_{i=1}^ka_{i,j}x_{i,j}\right)\right| \leq (1+\de)\max_{1\leq i\leq k}\sum_{j=1}^l|a_{i,j}|,
\end{equation}
for any $\de$ satisfying
\begin{equation}
\label{this long delta}
 \de \geq \sum_{i=1}^k \max\left\{12\e m_{t_i},12\frac{1}{m_{t_i}},6\frac{1}{N}m_{t_i},6\frac{m_{t_i}}{m_{t_i+1}}\right\}.
\end{equation}
\end{lem}


\begin{rem}
We point out that the vectors $x_{i,j}$, $1\leq i \leq k, 1\leq j\leq l$, are not required to have successive or disjoint supports.
\end{rem}

\begin{proof}[Proof of Lemma \ref{upper auxiliary}]
The proof is performed by induction on $m = 0,1,\ldots$ by showing that \eqref{upper auxiliary eq} holds for every $f\in W_m^N$. For $m = 0$ the result easily follows from the fact that for all $n\in\mathbb{N}$ and $1\leq i\leq k$, $1\leq j\leq l$ we have $|e_n^*(x_{i,j})| \leq m_{t_i}\e$ which yields
\[\left|e_n^*\left(\sum_{j=1}^l\sum_{i=1}^ka_{i,j}x_{i,j}\right)\right| \leq \left(\e \sum_{i=1}^km_{t_i}\right)\max_{1\leq i\leq k}\sum_{j=1}^l|a_{i,j}|.\]

Assume that the conclusion holds for every $f\in W_m^N$ and let $f\in W_{m+1}^N\setminus W_{m}^N$. Write
$$f = \frac{1}{m_{j_1}\cdots m_{j_a}}\sum_{q=1}^df_q$$
where $j_1,\ldots,j_a\in\N$ and $(f_q)_{q=1}^d$ is an $N$-sufficiently large and $\mathcal{S}_{n_{j_1}+\cdots+n_{j_a}}\ast\mathcal{A}_3$-admissible sequence of functionals in $W_m^N$. We define $b = \max\{j_1,\ldots,j_a\}$. The inductive assumption yields
\begin{equation}
\label{upper auxiliary proof eq1}
 \left|f_1\left(\sum_{j=1}^l\sum_{i=1}^ka_{i,j}x_{i,j}\right)\right| \leq (1+\de)\max_{1\leq i\leq k}\sum_{j=1}^l|a_{i,j}|.
\end{equation}
Set $B = \{2\leq q\leq  d: f_q = \pm e_n^* \text{ for some }n\in\N\}$ and $C = \{2,\ldots,d\}\setminus B$. Define
$$g_1 = \frac{1}{m_{j_1}\cdots m_{j_a}}f_1,\; g_2 = \frac{1}{m_{j_1}\cdots m_{j_a}}\sum_{q\in B}f_q, \text{ and } g_3 = \frac{1}{m_{j_1}\cdots m_{j_a}}\sum_{q\in C}f_q.$$
Clearly, $f = g_1 + g_2 + g_3$. It follows from the definition of $N$-sufficiently large that $\|g_3\|_\infty\leq 1/(Nm_{j_1}\cdots m_{j_a})$ which implies that for all $1\leq i \leq k, 1\leq j\leq l$ we have $|g_3(x_{i,j})|\leq m_{t_i}/(Nm_{j_1}\cdots m_{j_a})$ and hence
\begin{equation}
\begin{split}
\left|g_3\left(\sum_{j=1}^l\sum_{i=1}^ka_{i,j}x_{i,j}\right)\right| &\leq \left(\frac{1}{Nm_{j_1}\cdots m_{j_a}}\sum_{i=1}^km_{t_i}\right)\max_{1\leq i\leq k}\sum_{j=1}^l|a_{i,j}|\\
&\leq \frac{\de}{6}\max_{1\leq i\leq k}\sum_{j=1}^l|a_{i,j}|.
\end{split}
\end{equation}
Lemma \ref{basis on basis} yields that if we set $D = \{1\leq i\leq k:t_i\neq b\}$ then
\begin{equation*}
\begin{split}
\left|g_2\left(\sum_{j=1}^l\sum_{i\in D}^na_{i,j}x_{i,j}\right)\right| &\leq
\sum_{j=1}^l\sum_{i\in D}|a_{i,j}|\max\left\{2\e m_{t_i},\frac{m_{t_i}}{m_{t_i+1}},\frac{1}{m_{t_i}}\right\}\\
&\leq
\frac{\de}{6}\max_{1\leq i\leq k}\sum_{j=1}^l|a_{i,j}|,\end{split}
\end{equation*}
whereas an easy computation yields that if there is $1\leq i_0\leq k$ with $b = t_{i_0}$ then for all $1\leq j\leq l$ we have $|g_2(x_{i_0,j})|\leq 1$ and hence
\begin{equation}
\left|g_2\left(\sum_{j=1}^la_{i_0,j}x_{i_0,j}\right)\right| \leq \sum_{j=1}^l|a_{i_0,j}| \leq \max_{1\leq i\leq k}\sum_{j=1}^l|a_{i_0,j}|.
\end{equation}
We now have all the necessary components to complete the inductive step. We consider two cases, namely one in which such an $i_0$ does not exist  (i.e. when $D = \{1,\ldots,k\}$) and one in which such an $i_0$ exists  (i.e. $b = t_{i_0}$ for some  $1\leq i_0\leq k$). In the first case we obtain
\begin{align*}
\left|f\left(\sum_{j=1}^l\sum_{i=1}^ka_{i,j}x_{i,j}\right)\right|\span\\
&\leq \left|g_1\left(\sum_{j=1}^l\sum_{i=1}^ka_{i,j}x_{i,j}\right)\right| + \left|g_2\left(\sum_{j=1}^l\sum_{i=1}^ka_{i,j}x_{i,j}\right)\right| + \left|g_3\left(\sum_{j=1}^l\sum_{i=1}^ka_{i,j}x_{i,j}\right)\right|\\
& = \left|g_1\left(\sum_{j=1}^l\sum_{i=1}^ka_{i,j}x_{i,j}\right)\right| + \left|g_2\left(\sum_{j=1}^l\sum_{i\in D}ka_{i,j}x_{i,j}\right)\right| + \left|g_3\left(\sum_{j=1}^l\sum_{i=1}^ka_{i,j}x_{i,j}\right)\right|\\
&\leq \left(\frac{1+\de}{m_{j_1}\cdots m_{j_a}} + \frac{\de}{6} + \frac{\de}{6}\right)\max_{1\leq i\leq k}\sum_{j=1}^l|a_{i,j}| \leq (1+\de)\max_{1\leq i\leq k}\sum_{j=1}^l|a_{i,j}|.
\end{align*}
In the second case
\begin{align*}
\left|f\left(\sum_{j=1}^l\sum_{i=1}^ka_{i,j}x_{i,j}\right)\right|\span\\
&\leq \left|g_1\left(\sum_{j=1}^l\sum_{i=1}^ka_{i,j}x_{i,j}\right)\right| + \left|g_2\left(\sum_{j=1}^l\sum_{i=1}^ka_{i,j}x_{i,j}\right)\right| + \left|g_3\left(\sum_{j=1}^l\sum_{i=1}^ka_{i,j}x_{i,j}\right)\right|\\
& = \left|g_1\left(\sum_{j=1}^l\sum_{i=1}^ka_{i,j}x_{i,j}\right)\right| + \left|g_2\left(\sum_{j=1}^l\sum_{i\in D}a_{i,j}x_{i,j}\right)\right|  + \left|g_2\left(\sum_{j=1}^la_{i_0,j}x_{i_0,j}\right)\right|\\
&+ \left|g_3\left(\sum_{j=1}^l\sum_{i=1}^ka_{i,j}x_{i,j}\right)\right|\text{ (use $m_{j_1}\cdots m_{j_a} \geq m_{t_{i_0}}$)}\\
& \leq \left(\frac{1+\de}{m_{t_{i_0}}} + \frac{\de}{6} + 1 +\frac{\de}{6}\right)\max_{1\leq i\leq k}\sum_{j=1}^l|a_{i,j}| \text{ (use $\delta\geq 6/m_{t_{i_0}}$)}\\
&\leq\left(1 +3\frac{\de}{6} + \frac{\de}{2}\right)\max_{1\leq i\leq k}\sum_{j=1}^l|a_{i,j}|.
\end{align*}
The proof is complete
\end{proof}

\section{Rapidly increasing sequences and the basic inequality}
\label{ris and basic section}
Rapidly increasing sequences appear in every HI-type construction and this case is no different as the definition below follows the line of classical examples such as \cite{AH}. The basic inequality on such sequences is the main tool used to bound the norm of such vectors from above by the norm of vectors in the auxiliary spaces. To achieve the isometric representation of unconditional sequences as block sequences in subspaces of $\X$ we give a rather tight estimate in the basic inequality \eqref{this is the basic inequality in the flesh}.

\begin{defn}
\label{definition ris}
Let $C\geq 1$, $I$ be an interval of $\N$ and $(j_i)_{i\in I}$ be a strictly increasing sequence of natural numbers. A block sequence $(x_i)_{i\in I}$ is called a $(C,(j_i)_{i\in I})$ rapidly increasing sequence (RIS) if the following are satisfied.
\begin{itemize}
 \item[(i)] For all $i\in I$ we have $\|x_i\| \leq C$,
 \item[(ii)] for $i\in I\setminus\{\min(I)\}$ we have $\max\supp(x_{i-1}) < \sqrt{m_{j_i}}$, and
 \item[(iii)] $|f(x_i)| \leq C/w(f)$ for every $i\in I$ and $f\in\W$ with $w(f) < m_{j_i}$.
\end{itemize}
\end{defn}

\begin{prop}[basic inequality]
\label{basic inequality}
Let $(x_i)_{i\in I}$ be a $(C,(j_i)_{i\in I})$-RIS, $(a_i)_{i\in I}$ be a sequence of scalars, and $N < \min\{m_{j_{\min(I)}}, \min\supp(x_{\min(I)})\}$ be a natural number. Then, for every $f\in\W$ there exist $h\in\{\pm e_i^*:i\in\N\}\cup\{0\}$ and $g\in W_\mathrm{aux}^N$ with $w(f) = w(g)$ so that if $t_i = \max\supp(x_i)$ for $i\in I$ then we have
\begin{equation}
\label{this is the basic inequality in the flesh}
\left|f\left(\sum_{i\in I}a_i x_i\right)\right| \leq C\left(1+\frac{1}{\sqrt{m_{j_{i_0}}}}\right)\left|(h + g)\left(\sum_{i\in I}a_ie_{t_i}\right)\right|.
\end{equation}
\end{prop}

\begin{proof}
We use Remark \ref{inductive construction norming} to prove the statement by induction on $n=0,1,\ldots$ for every $f\in W_n$ and every RIS. We shall also include in the inductive assumption that $\supp(h)$ and $\supp(g)$ are subsets of $\{t_i:i\in I\}$ as well as the following:
\begin{itemize}
 \item[(i)] either $h = 0$,
 \item[(ii)] or $h$ is of the form $\pm e_{t_{i_1}}^*$ for some $i_1\in I$,  $t_{i_1} < \min\supp(g)$, and $w(f) > N$.
\end{itemize}

For $n=0$ the result is rather straightforward so let us assume that the conclusion holds for every $f\in W_n$ and let $f\in W_{n+1}$.
Let
$$f = \frac{1}{m_{s_1}\cdots m_{s_l}}\sum_{q=1}^df_q$$
with $(f_q)_{q=1}^d$ being  an $\mathcal{S}_{n_{s_1}+\cdots+n_{s_l}}$ admissible and very fast growing sequence of weighted functionals in $W_n$. By perhaps omitting an initial interval of the $f_q$'s we may assume that $\max\supp(f_1) \geq \min\supp (x_1)$. This means that for all $1<q\leq d$ we have $w(f_q) > \max\supp (f_1) > N$. We shall use this near the end of the proof. Define
$$i_0 = \max\left\{i\in I: m_{s_1}\cdots m_{s_l}\geq m_{j_i}\right\},$$
if such an $i_0$ exists (we will treat the case in which such an $i_0$ does not exist slightly further below). In this case $w(f) = m_{s_1}\cdots m_{s_l}\geq m_{j_{i_0}} >N$. Choose $\min(I)\leq i_1\leq i_0$ that maximizes the quantity $|a_i|$ for $i$ in $\{\min(I),\ldots,i_0\}$ and set $h = \mathrm{sign}(f(a_{i_1}x_{i_1}))e_{i_1}^*$. If $i_0>\min(I)$ it is straightforward to check $\|\sum_{i<i_0}a_ix_i\|_\infty \leq C|a_{i_1}|$ and we use this to show
\begin{equation}
\label{basic inequality eq1}
\begin{split}
\left|f\left(\sum_{i\leq i_0}a_i x_i\right)\right| &\leq \max\supp(x_{i_0-1})\left\|\sum_{i < i_0}a_i x_i\right\|_\infty \frac{1}{w(f)} + \left|f(a_{i_0}x_{i_0})\right|\\
&\leq C\frac{\max\supp(x_{i_0-1})}{m_{j_{i_0}}}|a_{i_1}| + C|a_{i_1}| \leq C\left(1+\frac{1}{\sqrt{m_{j_{i_0}}}}\right)|a_{i_1}|\\
& = C\left(1+\frac{1}{\sqrt{m_{j_{i_0}}}}\right)\left|h\left(\sum_{i\in I}a_ie_{t_i}\right)\right|.
\end{split}
\end{equation}
If $i_0 = \min(I)$ we simply obtain $|f(\sum_{i\leq i_0}a_ix_i)| \leq C|a_{i_1}|$. In either case estimate \eqref{basic inequality eq1} holds.

If such an $i_0$ does not exist (i.e. when $w(f) < m_{j_{\min(I)}}$) then set $h = 0$ and we have no lower bound for $w(f)$. This is of no concern as such a restriction is not included in the inductive assumption when $h=0$.

Depending on whether the above $i_0$ exists or not define $\tilde I = \{i\in I:i>i_0\}$ or $\tilde I = I$. It remains to find $g\in W_\mathrm{aux}^N$ with $w(g) = w(f)$ and $\supp(g)\subset \{t_i:i\in\tilde I\}$ so that $|f(\sum_{i\in\tilde I}a_ix_i)| \leq C(1+1/m_{j_0})|g(\sum_{i\in \tilde I}a_ie_{t_i})|$. Define
\begin{align*}
A &= \left\{i\in\tilde I: \text{ there exists at most one } q \text{ with }\ran(x_i)\cap\ran(f_q)\neq\emptyset\right\},\\
I_q &= \left\{i\in A: \ran(f_q)\cap\ran(x_i)\neq\emptyset\right\}\text{ for }1\leq q\leq d,\\
D &= \{1\leq q\leq d: I_q\neq \emptyset\}\text{ and}\\
B &= \tilde I\setminus A.
\end{align*}
Observe that the $I_q$'s are pairwise disjoint intervals. Apply the inductive assumption for each $f_q$ with $q\in D$ and the $(C,(m_{j_i})_{i\in I_q})$ RIS $(x_i)_{i\in I_q}$ to find $h_q\in\{\pm e_{t_i}^*:i\in I_q\}\cup\{0\}$ and $g_q\in W_\mathrm{aux}^N$ satisfying the inductive assumption, in particular
\begin{equation*}
\left|f_q\left(\sum_{i\in I_q}a_i x_i\right)\right| \leq C\left(1+\frac{1}{\sqrt{m_{j_{i^q_0}}}}\right)\left|(h_q + g_q)\left(\sum_{i\in I_q}a_ie_{t_i}\right)\right|.
\end{equation*}
Using the above it is not hard to see that $h$ and
$$g = \frac{1}{m_{s_1}\cdots m_{s_l}}\left(\sum_{i\in B}\mathrm{sign}(f(a_ix_i))e^*_{t_i}+ \sum_{q=1}^dh_q +\sum_{q=1}^d g_q\right)$$
satisfy \eqref{this is the basic inequality in the flesh}. To complete the proof it remains to show that the vectors $(e^*_{t_i})_{i\in B}{}^\frown (h_q)_{q\in D}{}^\frown (g_q)_{q\in D}$ can be ordered to form an $\mathcal{S}_{n_{s_1}+\cdots+n_{s_l}}\ast\mathcal{A}_3$ admissible and $N$-sufficiently large sequence.

For each $1\leq q \leq d$ we shall define a collection of at most three functionals $\mathcal{F}_q$ (it may also be empty) with the following properties:
\begin{itemize}
 \item[(a)] for each $\phi\in\mathcal{F}_q$ we have $\min\supp(f_q) \leq \min\supp(\phi)$ and if $1\leq q<d$ the $\max\supp(\phi) < \min\supp(f_{q+1})$
 \item[(b)] $\cup_{1\leq q\leq d}\mathcal{F}_q = \{e^*_{t_i}: i\in B\}\cup\{h_q:q\in D\}\cup\{g_q:q\in D\}$
\end{itemize}
For each $i\in B$ set $q_i = \max\{1\leq q\leq d: \min\supp(f_q)\leq \max\supp(x_i)\}$. Note that the correspondence $i\to q_i$ is strictly increasing. For each $q$ for which there is $i$ so that $q = q_i$ set $\mathcal{F}_q = \{h_q,g_q,e_{t_i}^*\}$. Depending on whether $q\in D$ and whether $h_q = 0$, some of the functionals $h_q$, $g_q$ may be omitted. For $q$ for which there is no $i$ with $q=q_i$ define $\mathcal{F}_q= \{h_q,g_q\}$, omitting if necessary any of $h_q$ or $g_q$. Properties (a) and (b) are not very hard to show.

It now follows from (a) and the spreading property of the Schreier families that the set $\{\min\supp(h):h\in\cup_{1\leq q\leq d}\mathcal{F}_q$ is $\mathcal{S}_{n_{s_1}+\cdots+n_{s_l}}\ast\mathcal{A}_3\}$ admissible. It follows from (b) that ordering the functionals in $(e^*_{t_i})_{i\in B}{}^\frown (h_q)_{q\in D}{}^\frown (g_q)_{q\in D}$ according to the minimum of their supports they are $\mathcal{S}_{n_{s_1}+\cdots+n_{s_l}}\ast\mathcal{A}_3$ admissible.

We now show that the sequence is $N$ sufficiently large. Recall now that for all $q>1$ we have $w(f_q) > N$ and hence if $g_q$ is defined we have $w(g_q)>N$. It remains to show that if $g_1$ is defined and it does not appear first in the enumeration above then $w(g_1) > N$. For this to be the case, the set $\mathcal{F}_1$ must contain the functional $h_1\neq 0$. By the inductive assumption this means $w(g_1) = w(f_1) > N$ and the proof is complete.
\end{proof}

\subsection{Existence of rapidly increasing sequences}
\label{section ris existence}
As is the case in past constructions, rapidly increasing sequences are given by special convex combinations of normalized block vectors that are bounded from bellow. To achieve the desired isometric representation we show that this lower bound may be chosen arbitrarily close to one. We then show that such sequences can be chosen to be $C$-RIS for any $C>1$.

\begin{prop}
\label{very good ell1 vectors}
Let $Y$ be a block subspace of $X$. Then for every $n\in\N$, $\e$, and $\de>0$ there exists a $(n,\e)$ s.c.c. $x = \sum_{i=1}^mc_ix_i$ with $\|x\| > 1/(1+\de)$ where $x_1,\ldots,x_m$ are in the unit ball of $Y$.
\end{prop}

\begin{proof}
Towards a contradiction assume that the conclusion is false. That is, for all $\mathcal{S}_n$-admissible vectors $(x_i)_{i=1}^m$ in the unit ball of $Y$ so that the vector $x = \sum_{i=1}^mc_ix_i$ is a $(n,\e)$ s.c.c. we have $\|x\| \leq 1/(1+\de)$.

Start with a normalized block sequence $(x_i)_i$ in $Y$ and take a subsequence $(x_i^0)_i$ that satisfies the conclusion of Proposition \ref{uniform ell1}. Using the properties of $(m_j)$, $(n_j)_j$ fix $j\in\N$ with $n_j \geq n$ and
\begin{equation}
\label{very good ell1 vectors eq1}
\frac{\left(\left(1+\de\right)^{\frac{1}{n}}\right)^{n_j}}{m_j} \geq 2(1+\de). 
\end{equation}
Define inductively block sequences $(x^k_i)_i$ for $0\leq k\leq \lfloor n_j/n\rfloor$ satisfying.
\begin{itemize}
 \item[(i)] for each $i,k$ there is a subset $F_i^k$ of $\N$ so that  $(x_m^{k-1})_{m\in F_i^k}$ is $\mathcal{S}_n$ admissible and coefficients $(c_m^{k-1})_{m\in F_i^k}$ so that $\tilde x_i^k = \sum_{m\in F_i^k}c_m^{k-1}x_m^{k-1}$ is a $(n,\e)$ s.c.c.
 \item[(ii)] for each $i,k$ we set $x_i^k = (1+\de)\tilde x_i^k$.
\end{itemize}
Using the negation of the desired conclusion, it is straightforward to check by induction that $\|x_i^k\| \leq 1$ and that for $k\leq \lfloor n_j/n\rfloor$ each vector $x_i^k$ can be written in the form
$$x_i^k = (1+\de)^k\sum_{m\in G_i^k}d_m^kx_m^0$$
for some subset $G_i^k$ of $\N$ so that $(x_m^0)_{m\in G_i^k}$ is $\mathcal{S}_{nk}$ admissible and the coefficients satisfy $\sum_{m\in G_i^k}d_m^k = 1$. As the sequence satisfies the conclusion of Proposition \ref{uniform ell1} we deduce that  for $k = \lfloor n_j/n\rfloor$ we have $n_j - n < kn\leq n_j$
\begin{equation*}
1\geq \|x_i^k\| \geq \frac{(1+\de)^k}{2m_j}  > \frac{(1+\de)^{\frac{n_j}{n}}}{2m_j},
\end{equation*}
and therefore by \eqref{very good ell1 vectors eq1} $1\geq1+\de$ which is absurd.
\end{proof}

\begin{prop}
\label{scc are ris}
Let $x = \sum_{i=1}^mc_ix_i$ be a $(n,\e)$ s.c.c. with $\|x_i\| \leq  1$ for $1\leq i \leq m$ and $f\in\W$ with $\vec w(f) = (j_1,\ldots,j_l)$ so that $n_{j_1} +\cdots +n_{j_l} < n$. Then we have
\begin{equation*}
 |f(x)| \leq \frac{1+2\e w(f)}{w(f)}.
\end{equation*}
\end{prop}

\begin{proof}
Let $f = (1/m_{j_1}\cdots m_{j_l})\sum_{q=1}^df_q$ with $(f_q)_{q=1}^d$ $\mathcal{S}_{n_{j_1}+\cdots+n_{j_l}}$-admissible. Consider the subset of $\{1,\ldots,m\}$
$$A = \left\{i: \text{ there is at most one } 1\leq q\leq 	d \text{ with } \ran(x_i)\cap\ran(f_q)\neq\emptyset\right\}$$
and observe that for each $i\in A$ we have $|f(x_i)| \leq 1/(m_{j_1}\cdots m_{j_l})$ and hence
\begin{equation}
\label{scc are ris eq1}
\left|f\left(\sum_{i=1}^mc_ix_i\right)\right| \leq \frac{1}{m_{j_1}\cdots m_{j_l}}\sum_{i\in A}c_i + \sum_{i\notin A}c_i.
\end{equation}
Set $B = \{1,\ldots,m\}\setminus A$. By the shifting property of the Schreier families it follows that the vectors $(x_i)_{i\in B\setminus\{\min(B)\}}$ are $\mathcal{S}_{n_{j_1}+\cdots+n_{j_l}}$ admissible. As the singleton $\{x_1\}$ is $\mathcal{S}_1$ admissible we conclude that $\sum_{i\in B}c_i < 2\e$. Applying this to \eqref{scc are ris eq1} immediately yields the desired conclusion.
\end{proof}

\begin{cor}
\label{building ris}
Let $Y$ be a block subspace of $X$ and $C>1$. Then there exists an infinite $(C,(j_i)_{i})$-RIS $(x_i)_i$ in $Y$ with $\|x_i\| \geq 1$ for all $i\in\N$.
\end{cor}

\begin{proof}
We define the sequence $(x_i)_i$ inductively as follows. Fix $\de > 0$ with $1+\de<C$ and having chosen $x_1,\ldots,x_{i-1}$ choose $j_i$ with $\sqrt{m_{j_i}} > \max\supp(x_{i-1})$, choose a natural number $k_i$ with the property that for all $s_1,\ldots,s_l\in\N$ that satisfy $m_{s_1}\cdots m_{s_l} < m_{j_i}$ we have $n_{s_1}+\cdots+n_{s_l} < k_i$, and choose $\e_i > 0$ with $(1+\de)(1+2\e_im_i) \leq C$. Use Proposition \ref{very good ell1 vectors} to find an $(k_i,\e_i)$ s.c.c. $(y_i)$ in $Y$ with $\min\supp(y_i)>\max\supp(x_i)$ and $1/(1+\de)\leq \|y_i\| \leq 1$ and set $x_i = (1+\de)y_i$. Proposition \ref{scc are ris} yields that $(x_i)_i$ is the desired vector.
\end{proof}

\section{Hereditary Asymptotic structure of $\X$}
\label{fbr in X}
This section is devoted to the study of the asymptotic behavior of subspaces of $\X$. As it was shown in Section \ref{uniform unique spreading model section} the space $\X$ only admits spreading models 4-equivalent to the unit vector basis of $\ell_1$. We show that the joint behavior of arrays of sequences does not retain this uniform behavior. In fact, $c_0$ is an asymptotic model of every subspace of $\X$ and every 1-unconditional sequence is block finitely representable in every block subspace of $\X$. These results in particular yield that $\X$ does not have an asymptotic-$\ell_p$ subspace.

\begin{prop}
\label{omega joint spreading models}
Let $Y$ be a block subspace of $\X$ and $\varepsilon>0$. Then there exists an array of block sequences $(x_j^{(i)})_j$, $i\in\mathbb{N}$, in $Y$ so that for any $k,l\in\mathbb{N}$, scalars $(a_{i,j})_{1\leq i\leq k, 1\leq j\leq l}$, and plegma family $(s_i)_{i=1}^k$ in $[\mathbb{N}]^l$ with $\min(s_1) \geq \max\{k,l\}$ we have
\begin{equation}
\label{omega equation}
\max_{1\leq i\leq k}\sum_{j=1}^l|a_{i,j}| \leq \left\|\sum_{i=1}^k\sum_{j=1}^la_{i,j}x_{s_i(j)}^{(i)}\right\| \leq (1+\varepsilon)\max_{1\leq i\leq k}\sum_{j=1}^l|a_{i,j}|.
\end{equation}
\end{prop}

\begin{proof}
Fix $1 < C < \min\{(1+\e)^{1/4},2\}$ and $0<\de \leq ((1+\e)^{1/2}-1)/2$. Using the properties of the sequences $(m_j)_j$, $(n_j)_j$ from Section \ref{definition section}, page \pageref{definition section} we fix a sequence of pairwise different natural numbers $(t_i)_{i=1}^\infty$ satisfying for $i\in\mathbb{N}$
\begin{equation}
\frac{1}{m_{t_i}}\leq \frac{\de}{12\cdot 2^i}\text{ and } \frac{m_{t_i}}{m_{t_i+1}} \leq \frac{\de}{6\cdot 2^i}.
\end{equation}
For each $k\in\mathbb{N}$ fix $\bar \e_k >0$ and $N_k\in\mathbb{N}$ so that for $1\leq i \leq k$
\begin{equation}
\bar \e_k \leq \frac{\de}{12m_{t_i}2^i} \text{ and } \frac{m_{t_i}}{N_k} \leq \frac{\delta}{6\cdot 2^i}.
\end{equation}
Observe that for any $k\in\N$ we have that $\de$, $N_k$, $\bar\e_k,$ and $(m_{t_i})_{i=1}^k$ satisfy \eqref{this long delta}.

Use Corollary \ref{building ris} to find an infinite $(C,(\bar{j_s})_{s})$-RIS $(y_s)_s$ in $Y$ with $\|y_s\| \geq 1$ for all $s\in\N$. By perhaps passing to a subsequence we may assume that for all $s\in\N$ we have
\begin{equation}
\label{*how to find that one basis eqegg}
\begin{split}
&\frac{1}{\sqrt{m_{\bar{j_s}}}} \leq (1+\e)^{1/4} - 1,\\
&{N_s}\leq m_{\bar{j_s}}, \text{ and } \min\supp(y_s)\geq \max_{1\leq i\leq s}\{n_{t_i},N_i,6/\bar\e_i\}.
\end{split}
\end{equation}
For each $s$ find $f_s$ in $\W$ with $\supp(f_s)\subset \supp(y_s)$ and $f_s(y_s) = \|y_s\| \geq 1$. Note that for all $s$ we have $w(f_s) \geq m_{\tilde{j_s}}$, otherwise by Property (iii) of \ref{definition ris} we would have $1\leq f_s(y_s) \leq C/w(f_s) < 2/w(f_s) \leq 1$ (because $m_1 = 2$) which is absurd. Hence, using Property (ii) of \ref{definition ris}, for all $s>1$ we have $w(f_s) \geq m_{\tilde{j_s}} \geq (\max\supp(y_{s-1}))^2 \geq (\max\supp(f_{s-1}))^2 > \max\supp(f_{s-1})$, i.e. $(f_s)_s$ is very fast growing.

Choose disjoint finite subsets of $\mathbb{N}$, $F^{(i)}_j$, $i,j\in\mathbb{N}$, so that for each $i,j\in\mathbb{N}$ we have $F^{(i)}_j < F^{(i)}_{j+1}$ and $\{\min\supp(y_s): s\in F^{(i)}_{j}\}$ is a maximal $\mathcal{S}_{n_{t_i}-1}$. Using Proposition \ref{basic scc exist in abundance} find coefficients $(c_{s}^{i,j})_{s\in F^{(i)}_{j}}$ so that the vector $\tilde x_{i,j} = \sum_{s\in F^{(i)}_{j}}c_s^{i,j}y_s$ is an $(n_{j_i}-1,\bar\e_j/2)$ s.c.c. Note that by Remark \ref{some remarks for the far future 2} if $\phi_s = \max\supp(y_s)$ then the vector $\tilde z_{i,j} = \sum_{s\in F^{(i)}_{j}}c_s^{i,j}e_{\phi_s}$ is a $(n_{j_i}-1,\bar\e_j)$ basic s.c.c. Hence, for any $k,l\in\mathbb{N}$ and $k \leq s_i(1) < \cdots<s_i(l)$, for $1\leq i\leq k$ the vectors $z^{(i)}_{s_i(j)} = m_{t_i}\tilde z_{i,s_i(j)}$, $1\leq i\leq k$ $1\leq j\leq l$ satisfy \eqref{upper auxiliary eq} of Lemma \ref{upper auxiliary} with the $\delta$, $N_k$, $\bar \varepsilon_k$ chosen above.

Define $x^{(i)}_{j} = m_{t_i}\tilde x_{i,j}$ for $i,j\in\mathbb{N}$. We will show that this is the desired sequence and to that end let $k,l\in\mathbb{N}$ and let $(s_i)_{i=1}^k$ be a plegma in $[\mathbb{N}]^l$ with $\min(s_1)\geq \max\{k,l\}$. For the upper inequality, Proposition \ref{basic inequality} yields that for any scalars $(a_{i,j})_{1\leq i\leq k, 1\leq j\leq l}$ we have
\begin{align*}
\left\|\sum_{j=1}^l\sum_{i=1}^ka_{i,j}x^{(i)}_{s_i(j)}\right\|\leq\span \\
\leq& C\left(1+ \frac{1}{\sqrt{m_{\bar{j_1}}}}\right)\left(\max_{\substack{1\leq i\leq k\\1\leq j\leq l}}\max_{s\in F^{(i)}_{j}}\left(m_{t_i}|a_{i,j}|c^{i,j}_s\right) +\left\|\sum_{j=1}^l\sum_{i=1}^ka_{i,i}z^{(i)}_{s_i(j)}\right\|_{\mathrm{aux},N_k}\right)\\
\leq& (1+\e)^{1/4}\left(1+\e\right)^{1/4}\left(\max_{\substack{1\leq i\leq k\\1\leq j\leq l}}\left(m_{t_i}|a_{i,j}|\bar\varepsilon_k\right)+\left\|\sum_{j=1}^l\sum_{i=1}^ka_{i,i}z^{(i)}_{s_i(j)}\right\|_{\mathrm{aux},N_k}\right)\\
\leq& (1+\e)^{1/2}\left(\de\max_{1\leq i\leq k}\sum_{j=1}^l|a_{i,j}|+\left\|\sum_{j=1}^l\sum_{i=1}^ka_{i,i}z^{(i)}_{s_i(j)}\right\|_{\mathrm{aux},N_k}\right)\\
\leq& (1+\e)^{1/2}\left(\de\max_{1\leq i\leq k}\sum_{j=1}^l|a_{i,j}|+(1+\de)\max_{1\leq i\leq k}\sum_{j=1}^l|a_{i,j}| \right)\text{ (from \eqref{upper auxiliary eq})}\\
\leq& (1+\e)^{1/2}(1+2\de)\max_{1\leq t\leq n}\sum_{s=1}^n|a_{s,t}| \leq (1+\e)\max_{1\leq t\leq n}\sum_{s=1}^n|a_{s,t}|.
\end{align*}
For the lower inequality we observe that for fixed $1\leq i_0\leq n$ the functionals $((f_s)_{s\in F^{(i_0)}_{s_{i_0}(j)}})_{j=1}^l$ are very fast growing and for each $1\leq j\leq l$ the functionals $(f_s)_{s\in F^{(i_0)}_{s_{i_0}(j)}}$ are $\mathcal{S}_{n_{t_{i_0}}-1}$ admissible. It follows from \eqref{*how to find that one basis eqegg} that $((f_s)_{s\in F^{(i_0)}_{s_{i_0}(j)}})_{j=1}^l$ is $\mathcal{S}_{n_{t_{i_0}}}$-admissible and hence $f = (1/m_{t_{i_0}})\sum_{j=1}^l\sum_{s\in F^{(i_0)}_{i_0,j}}f_s$ is in $\W$. It follows that $f(x^{(i_0)}_{s_{i_0}(j)})\geq 1$ for all $1\leq j\leq l$ which means that for any coefficients $(a_{i,j})_{1\leq i\leq k, 1\leq j\leq l}$ we have
\begin{equation*}
\begin{split}
 \left\|\sum_{j=1}^l\sum_{i=1}^ka_{i,j}x^{(i)}_{j}\right\| &=  \left\|\sum_{j=1}^l\sum_{i=1}^k|a_{i,j}|x^{(i)}_{j}\right\| \geq f\left(\sum_{j=1}^l\sum_{i=1}^k|a_{i,j}|x^{(i)}_{j}\right)\\
  &= f\left(\sum_{j=1}^l|a_{i_0,j}|x^{(i_0)}_{j}\right)\geq \sum_{j=1}^l|a_{i_0,j}|.
\end{split}
\end{equation*}
\end{proof}

\begin{thm}
\label{c0 asmodel}
Let $Y$ be a block subspace of $\X$.
\begin{itemize}

\item[(a)] For every $\varepsilon>0$ there exists an array of block sequences in $Y$ that generate an asymptotic model that is $(1+\varepsilon)$-equivalent to the unit vector basis of $c_0$.

\item[(b)] For every $\varepsilon>0$ and $k\in\mathbb{N}$ there exists a $k$-array of block sequences in $Y$ that generate a joint spreading model $(1+\varepsilon)$-equivalent to the basis of $\ell_\infty^k(\ell_1)$.
\end{itemize}
In particular, $X$ does not contain an asymptotic-$\ell_1$ subspace.
\end{thm}

\begin{proof}
Let $(x_j^{(i)})_j$, $i\in\mathbb{N}$ be the infinite array given by Proposition \ref{omega joint spreading models}, for some fixed $\varepsilon >0$. Then, it easily follows that this infinite array generates the unit vector basis of $c_0$ as a spreading model. This is because the asymptotic model is witnessed by taking one vector from each sequence. It is entirely immediate by the definition of joint spreading models that the first $k$ sequences in the array generate the basis of $\ell_\infty^k(\ell_1)$ as a joint spreading model.
\end{proof}

\begin{cor}
\label{hereditary fbr}
Let $Y$ be a block subspace of $\X$. Every 1-unconditional basic sequence is finitely block representable in $Y$. In fact, for every $k\in\mathbb{N}$ every $k$-dimensional space with a 1-unconditional basis is an asymptotic space for $Y$, in the sense of \cite{MMT}.
\end{cor}

\begin{proof}
By Proposition \ref{universal for unc} it is sufficient to show that the sequence $(e_{i,j})_{i,j=1}^n$ mentioned in the statement of that result, with the lexicographical order, is an asymptotic space for $Y$. Fix $\varepsilon >0$ and let $(x_j^{(i)})_j$, $i\in\mathbb{N}$ be the infinite array given by Proposition \ref{omega joint spreading models}. It is an easy observation that for a sufficiently sparsely chosen strict plegma $(s_j)_{j=1}^n$ in $[\mathbb{N}]^n$  that the sequence $(x^{(j)}_{s_j(i)})_{i,j=1}^n$ is a block sequence with the lexicographical order. Moreover, if $\min(s_1) \geq n$ then $(x^{(j)}_{s_j(i)})_{i,j=1}^n$ is $(1+\varepsilon)$-equivalent to $(e_{i,j})_{i,j=1}^n$.
\end{proof}

\begin{cor}
\label{krivine set and refl}
Let $Y$ be a block subspace of $\X$. Then $K(Y) = [1,\infty] \supsetneq \{1\} = \widetilde K(Y)$. Furthermore, $\ell_1$ and $c_0$ don't embed into $\X$, hence $\X$ is reflexive.
\end{cor}

Reflexivity and Proposition \ref{unique 1 or 0 am} yield the following (see Definition \cite{defas}).
\begin{cor}
The space $\X$ is asymptotically symmetric.
\end{cor}

\begin{rem}
The construction of $\X$ can be modified to obtain  for any $1\leq p <\infty$ a Banach space a space $X_\mathrm{iw}^p$. One takes a norming $W_\mathrm{iw}^p$ so that for any $\mathcal{S}_{n_{j_1}+\cdots+n_{j_l}}$ admissible sequence of very fast growing functionals $f_1<\cdots<f_d$ and any $(c_q)_{q=1}^d$ in the unit ball of $\ell_{p'}$ the function $f = (1/m_{j_1}\cdots m_{j_l})\sum_{q=1}^dc_qf_q$ is in $W_\mathrm{iw}^p$ as well. It is completely natural to expect that similar techniques will yield that this space has a unique and uniform $\ell_p$ spreading model, $c_0$ is an asymptotic model of every subspace, and the Krivine set of every subspace of $X_\mathrm{iw}^p$ is $[p,\infty]$. This modification does not apply to the case $p=\infty$. To obtain a space with a unique and uniform $c_0$ spreading model without an asymptotic-$c_0$ subspace we must look at the dual of $\X$ and this is the subject of Section \ref{dual section}.
\end{rem}

\section{The spaces $\X^p$, $1<p<\infty$}
We describe how the construction of $\X$ can be modified to obtain a space with a uniformly unique $\ell_p$-spreading model, where $1<p<\infty$, and a $c_0$-asymptotic model in every subspace. We give the steps that need to be followed in order to reach the conclusion but we omit most proofs as they are in the spirit of $\X$.

We fix a $p\in(1,\infty)$ and we denote by $p^*$ its conjugate. Given a subset $G$ of $c_{00}(\N)$, $j_1,\ldots,j_l\in\N$, real numbers $(\la_q)_{q=1}^d$  with $\sum_{q=1}^d|\la_q|^{p^*}\leq 1$, and $f_1 < \cdots < f_d$ in $G$ that are $\mathcal{S}_{n_{j_1}+\cdots+n_{j_l}}$-admissible we call a functional of the form
\begin{equation*}
f = \frac{1}{m_{j_1}\cdots m_{j_l}}\la_q\sum_{q = 1}^d f_q
\end{equation*}
a weighted functional of $G$ of weight $w(f) = m_{j_1}\cdots m_{j_l}$ and vector weight $\vec w(f) = (j_1,\ldots,j_l)$. For all $i\in\N$, we also call $f = \pm e_i^*$ a weighted functional of weight $w(f) = \infty$. We define very fast growing sequences as in Section \ref{definition section}. We then let $\W^p$ be the smallest subset of $c_{00}(\N)$ that satisfies the following two conditions.
\begin{itemize}

\item[(i)] $\pm e_i^*$ is in $\W^p$ for all $i\in\N$ and

\item[(ii)] for every $j_1,\ldots,j_l\in\N$, real numbers $(\la_q)_{q=1}^d$  with $\sum_{q=1}^d|\la_q|^{p^*}\leq 1$, and every $\mathcal{S}_{n_{j_1}+\cdots+n_{j_l}}$-admissible and very fast growing sequence of weighted functionals $(f_q)_{q=1}^d$ in $\W^p$ the functional $$f = \frac{1}{m_{j_1}\cdots m_{j_l}}\sum_{q=1}^d\la_qf_q$$ is in $\W^p$.

\end{itemize}
Set $\X^p$ to be the space defined by this norming set.

The following is  similar to \cite[Proposition 2.9]{DM} and \cite[Proposition 4.2]{BFM}. We give a short proof.
\begin{prop}
\label{upper p}
Let $(x_i)_{i=1}^n$ be a a normalized block sequence in $\X^p$. Then for any scalars $c_1,\ldots,c_n$ we have
\begin{equation}
\label{always upper ellp}
\left\|\sum_{i=1}^na_ix_i\right\| \leq 2 \left(\sum_{i=1}^n|a_i|^p\right)^{1/p}
\end{equation}
\end{prop}

\begin{proof}
This is proved by induction on $m$ with $\W^p = \cup_{m=0}^\infty W_m$. Assume that for every $f\in W_m$, every normalized block vectors $x_1<\cdots<x_n$, and every scalars $c_1,\ldots,c_n$ with $(\sum|c_j|^p)^{1/p}\leq 1$ we have $|f(c_1x_1+\cdots +c_nx_n)| \leq 2$. Let now $f = (1/m_j\cdots m_{j_l})\sum_{q=1}^d\la_qf_q$ be in $W_{m+1}$ with $f_1,\ldots,f_d\in W_m$, $(x_j)_{j=1}^n$ be a normalized block sequence, and $(c_j)_{j=1}^l$ be scalars with $(\sum|c_j|^p)^{1/p}\leq 1$. Set $x = \sum_{i=1}^nc_ix_i$. Define the sets
\[
\begin{split}
D_j &= \{i: \supp (f_i)\cap \supp(x_j)\neq\varnothing\}, \text{ for }j=1,\ldots,n\\
E_j &= \{i\in D_j: j = \min\{j': i\in D_{j'}\}\},\text{ for }j=1,\ldots,n,\\
F_j &= D_j\setminus E_j, \text{ for }j=1,\ldots,n,\text{ and }\\
G_i &= \{j: i\in F_j\},\text{ for }i=1,\ldots,d.
\end{split}
\]
Observe that the sets $(E_j)_{j=1}^n$ are pairwise disjoint and the sets $(G_i)_{i=1}^d$ are pairwise disjoint as well. For $j=1,\ldots,n$ set $\Lambda_j = (\sum_{i\in E_j}|\la_i|^{p^*})^{1/p^*}$ and for $i=1,\ldots,d$ set $C_i = (\sum_{j\in G_i}|c_j|^p)^{1/p}$. Then,
\[
\begin{split}
|f(x)| &= \left|\sum_{j=1}^mc_j\Lambda_j\!\!\left(\frac{1}{m_j\cdots m_{j_l}}\!\sum_{i\in E_j}\frac{\la_i}{\Lambda_j}f_i\right)\!\!(x_j) + \frac{1}{m_j\cdots m_{j_l}}\!\sum_{j=1}^nc_j\!\!\sum_{i\in F_j}\la_if_i(x_j)\right|\\
&\leq \left(\sum_{j=1}^n|c_j|^p\right)^{1/p}\!\!\!\left(\sum_{j=1}^n\Lambda_j^{p^*}\right)^{1/p^*}\!\!\!\! + \frac{1}{2}\sum_{i=1}^d|\la_i|\left|f_i\left(\sum_{j\in G_i}c_jx_j\right)\right|\\
&\leq 1+ \frac{1}{2}\sum_{i=1}^d|\la_i|2C_i\leq 1 + \left(\sum_{i=1}^d|\la_i|^{p^*}\right)^{1/p^*}\!\!\!\!\left(\sum_{i=1}^dC_i^p\right)^{1/p}\!\!\!\leq 2.
\end{split}
\]
\end{proof}

The proof of the following Proposition is practically identical to the proof of Proposition \ref{uniform ell1}
\begin{prop}
\label{uniform ellp}
Let $(x_i)_i$ be a normalized block sequence in $\X^p$. Then there exists $L\in[\N]^\infty$ so that for every $j_0\in\N$, every $F\subset L$ with $(x_i)_{i\in F}$ being $\mathcal{S}_{n_{j_0}}$-admissible, and every scalars $(c_i)_{i\in F}$ we have
\begin{equation*}
\left\|\sum_{i\in F}c_ix_i\right\| \geq \frac{1}{2m_{j_0}}\left(\sum_{i\in F}|c_i|^p\right)^{1/p}. 
\end{equation*}
In particular, every normalized block sequence in $\X$ has a subsequence that generates a spreading model that is 8-equivalent to the unit vector basis of $\ell_p$.
\end{prop}

The auxiliary spaces are each defined via collection of norming sets $W^{p,N}_{\mathrm{aux}}$, $N\in\N$. For each $N\in\N$ the set $W^{p,N}_{\mathrm{aux}}$ contains all
\[f = \frac{2^{1/p^*}}{m_{j_i}\cdots m_{j_l}}\sum_{q=1}^d\la_qf_q,\]
where $(f_q)_{q=1}^d$ is a sequence of $\mathcal{S}_{n_{j_1}+\cdots+n_{j_l}}\ast\mathcal{A}_3$-admissible functionals in $W^{p,N}_{\mathrm{aux}}$ so that for $q\geq 2$ we have $w(f_q) > N$ and $(\la_q)_{q=1}^d$ satisfy $\sum_{q=1}^d|\la_q|^{p^*}\leq 1$. The factor $2^{1/p^*}$ is necessary to prove the basic inequality and it also appears in \cite[Section 3]{DM}.

Recall from \cite[Section 3]{DM} that a vector $x = \sum_{i\in F}a_ie_i$ is called a $(n,\e)$ basic special $p$-convex combination (or basic s.$p$-c.c.) if $a_i\geq 0$, for $i\in F$, and $\sum_{i\in F}a_i^pe_i$ is a $(n,\e^p)$ basic s.c.c. The proof of the following is in the spirit of the proof of Lemma \ref{basis on basis} and Lemma \ref{upper auxiliary}

\begin{lem}
\label{p-upper auxiliary}
Let $\delta>0$. Then there exists $M\in\N$ so that for any $k\in\N$, any pairwise different natural numbers $(t_i)_{i=1}^k$ with $t_i\geq M$, for any $l\in\N$  and $\e>0$, there exists $N\in\N$, so that for  any vectors $(x_{i,j})_{1\leq i \leq k, 1\leq j\leq l}$ of the form
\begin{equation}
x_{i,j} = \frac{m_{t_i}}{2^{1/p^*}}\tilde x_{i,j},\text{ where } \tilde x_{i,j} = \sum_{r\in F_{i,j}}c_r^{i,j}e_r \text{ is a } (n_{t_i},\e) \text{ basic s.$p$-c.c.,}
\end{equation}
$1\leq i \leq k, 1\leq j\leq l,$ any scalars $(a_{i,j})_{1\leq i \leq k, 1\leq j\leq l}$, and any $f\in W^{p,N}_{\mathrm{aux}}$ we have
\begin{equation}
\label{upper auxiliary eq}
\left|f\left(\sum_{j=1}^l\sum_{i=1}^ka_{i,j}x_{i,j}\right)\right| \leq (1+\de)\max_{1\leq i\leq k}\left(\sum_{j=1}^l|a_{i,j}|^p\right)^{1/p}.
\end{equation}
\end{lem}

RIS are defined exactly like Definition \ref{definition ris}. The basic inequality is slightly different to Proposition \ref{basic inequality}.

\begin{prop}
\label{p-basic inequality}
Let $(x_i)_{i\in I}$ be a $(C,(j_i)_{i\in I})$-RIS, $(a_i)_{i\in I}$ be a sequence of scalars, and $N < \min\{m_{j_{\min(I)}}, \min\supp(x_{\min(I)})\}$ be a natural number. Then, for every $f\in\W^p$ there exist $h\in\{\pm e_i^*:i\in\N\}\cup\{0\}$, $g\in W_\mathrm{aux}^{p,N}$ with $w(f) = w(g)$, and $\la$, $\mu$ with $|\la|^{p^*} + |\mu|^{p^*} \leq 1$, so that if $t_i = \max\supp(x_i)$ for $i\in I$ then we have
\begin{equation}
\label{this is the basic inequality in the flesh}
\left|f\left(\sum_{i\in I}a_i x_i\right)\right| \leq C\left(1+\frac{1}{\sqrt{m_{j_{i_0}}}}\right)\left|(\la h + \mu g)\left(\sum_{i\in I}a_ie_{t_i}\right)\right|.
\end{equation}
\end{prop}

Using Proposition \ref{upper p} and Proposition \ref{uniform ellp} one can perform an argument similar to that in the proof of Proposition \ref{very good ell1 vectors} to show that every block sequence in $\X^p$ has a further block sequence, with norm at least $(1-\delta)$, that is a $(2+\varepsilon)$-RIS. The next result is similar to Proposition \ref{omega joint spreading models}.

\begin{prop}
\label{p omega joint spreading models}
Let $Y$ be a block subspace of $\X^p$. Then there exists an array of block sequences $(x_j^{(i)})_j$, $i\in\mathbb{N}$, in $Y$ so that for any $k,l\in\mathbb{N}$, scalars $(a_{i,j})_{1\leq i\leq k, 1\leq j\leq l}$, and plegma family $(s_i)_{i=1}^k$ in $[\mathbb{N}]^l$ with $\min(s_1) \geq \max\{k,l\}$ we have
\begin{equation}
\label{omega equation}
\frac{1}{2^{1/p^*}}\max_{1\leq i\leq k}\left(\sum_{j=1}^l|a_{i,j}|^p\right)^{1/p}\!\!\!\!\! \leq \left\|\sum_{i=1}^k\sum_{j=1}^la_{i,j}x_{s_i(j)}^{(i)}\right\| \leq 3\max_{1\leq i\leq k}\left(\sum_{j=1}^l|a_{i,j}|^p\right)^{1/p}\!\!\!\!.
\end{equation}
\end{prop}

The main result of this section follows in the same manner as Theorem \ref{c0 asmodel}

\begin{thm}
\label{c0 asmodel p}
Let $Y$ be a block subspace of $\X^p$.
\begin{itemize}

\item[(a)] There exists an array of block sequences in $Y$ that generate an asymptotic model that is $6$-equivalent to the unit vector basis of $c_0$.

\item[(b)] For every $k\in\mathbb{N}$ there exists a $k$-array of block sequences in $Y$ that generate a joint spreading model $6$-equivalent to the basis of $\ell_\infty^k(\ell_p)$.
\end{itemize}
In particular, $X$ does not contain an asymptotic-$\ell_p$ subspace.
\end{thm}

It is not true that all unconditional bases are finitely block representable in every subspace of $\X^p$. However the following is true. 
\begin{cor}
\label{krivine set p}
For every block subspace $Y$ of $\X^p$ the Krivine set of $Y$ is $K(Y) = [p,\infty]$. In fact, for every $q\in[p,\infty]$ the unit vector basis of $\ell_q^k$ is and asymptotic space for $Y$.
\end{cor}

\begin{proof}
The  inclusion $K(Y) \subset [p,\infty]$ is an immediate consequence of Proposition \ref{upper p}. To show the inverse inclusion we observe that by Theorem \ref{c0 asmodel p} (ii) for every $n\in\mathbb{N}$ the sequence $(e_{i,j})_{j=1}^n$, with the lexicographical order, endowed with the norm
\[\left\|\sum_{i,j}a_{i,j}e_{i,j} = \right\|\max_{1\leq i\leq k}\left(\sum_{j=1}^l|a_{i,j}|^p\right)^{1/p}\]
is an asymptotic space for $Y$, up to a constant 6.

A proof similar to Proposition \ref{universal for unc} gives that for any $\varepsilon>0$, $k\in\mathbb{N}$, and $p\leq q\leq \infty$ there is $n\in\mathbb{N}$ so that the unit vector basis of $\ell_q^k$ is $(1+\varepsilon)$-block representable in $(e_{i,j})_{j=1}^n$. To see this one needs to use the fact that for $p<q<\infty$ if we set $r = (qp)/(q-p)$ then
\[\left(\sum_{i=1}^k|a_i|^q\right)^{1/q} = \sup\left\{\left(\sum_{i=1}^k|a_ib_i|^p\right)^{1/p}: \left(\sum_{i=1}^k|b_i|^r\right)^{1/r}\leq 1\right\}.\]
The above follows from a simple application of H\"older's inequality.
\end{proof}

\begin{rem}
Because $\X^p$ has a uniformly unique $\ell_p$-spreading model the strong Krivine set of every block subspace of $\X^p$ is the singleton $\{p\}$.
\end{rem}

\section{The space $\X^*$}\label{dual section}
\label{dual section}
In this section we study the space $\X^*$. We prove that every normalized block sequence in $\X^*$ has a subsequence that generates a spreading model that is 4-equivalent to the unit vector basis of $c_0$. In addition, every block subspace of $\X^*$ admits the unit vector basis of $\ell_1$ as an asymptotic model and hence $\X^*$ does not have an asymptotic-$c_0$ subspace.

\begin{lem}
\label{vfg covnex}
Let $j_0\in\mathbb{N}$, $(g_k)_{k=1}^m$ be an $\mathcal{S}_{n_{j_0}}$-admissible sequence in $\mathrm{co}(\W)$ and assume the following: each $g_k$ has the form  $g_k = \sum_{j=1}^{d_k} c_j^k f_j^k$, where $d_k\in\mathbb{N}$ and $f_j^k\in \W$, for $1\leq k \leq m$, so that
\[\min\{w(f_j^k):1\leq j\leq d_k\} > \max\supp(g_{k-1}), \text{ for } 2\leq k \leq m,\]
then we have that $(1/m_{j_0})\sum_{k=1}^mg_k$ is in $\mathrm{co}(\W)$.
\end{lem}

\begin{proof}
By repeating some entries we may assume that $d_k = d$ and $c_j^k = c_j$ for each $1\leq k\leq m$. That is, for each $1\leq k\leq m$, we may assume $g_k = \sum_{j=1}^dc_jf_j^k$, where perhaps some $f_j^k$'s are repeated and perhaps some are the zero functional. We can also assume that $\supp(f_j^k\subset\supp(g_k))$, for $1\leq k\leq m$ and $1\leq j\leq d$. We conclude that for $1\leq j\leq d$ the sequence $(f^k_j)_{k=1}^m$ is an $\mathcal{S}_{n_{j_0}}$-admissible and very fast growing sequence in $\W$, so $f_j = (1/m_{j_0})\sum_{k=1}^mf_j^k$ is in $\W$. We conclude that $(1/m_{j_0})\sum_{k=1}^mg_k = \sum_{j=1}^dc_jf_j$ is in $\mathrm{co}(\W)$.
\end{proof}

\begin{lem}
\label{same weight covnex}
Let $j_0\in\mathbb{N}$, $(g_k)_{k=1}^m$ be an $\mathcal{S}_{n_{j_0}}$-admissible sequence in $\mathrm{co}(\W)$ and assume the following: there is $(j_1,\ldots,j_l)\in\mathbb{N}^{<\infty}$ so that each $g_k$ has the form  $g_k = \sum_{j=1}^{d_k} c_j^k f_j^k$, where $d_k\in\mathbb{N}$ and $f_j^k\in \W$, for $1\leq k \leq m$, so that $\vec w(f_j^k) = (j_1,\ldots,j_l)$ and if
\[
f_j^k = \frac{1}{m_{j_1}\cdots m_{j_l}}\sum_{r\in F_j^k} h_r^{k,j},
\]
with $(h_r^{k,j})_{r\in F_j^k}$ being $\mathcal{S}_{n_{j_1}+\cdots+n_{j_l}}$-admissible and very fast growing,  then
\[\min\{w(h_r^{k,j}):r\in F_j^k\} > \max\supp(g_{k-1}), \text{ for } 2\leq k \leq m,\]
then we have that $(1/m_{j_0})\sum_{k=1}^mg_k$ is in $\mathrm{co}(\W)$.
\end{lem}

\begin{proof}
As in the proof of Lemma \ref{vfg covnex} we may assume that there are $d$ and $c_1,\ldots,c_d$ so that $g_k = \sum_{j=1}^{d} c_j f_j^k$ where perhaps some $f_j^k$'s are repeated and perhaps some are the zero functional. It follows that for fixed $1\leq j\leq d$ the sequence $((h_r^{k,j})_{r\in F_j^k})_{k=1}^m$ is $\mathcal{S}_{n_{j_1}+\cdots+n_{j_l}+n_{j_0}}$-admissible and very fast growing. This means that $f_j = (1/m_{j_1}\cdots m_{j_l}m_{j_0})\sum_{k=1}^m\sum_{r\in F_j^k}h_r^{k,j}$ is in $\W$. We conclude that $(1/m_{j_0})\sum_{k=1}^mg_k = \sum_{j=1}^dc_jf_j$ is in $\mathrm{co}(\W)$.
\end{proof}

\begin{lem}
\label{getting rid of}
Let $(f_k)_k$ be a block sequence in $\mathrm{co}(\W)$ and let $\varepsilon > 0$. Then there exists $L\in[\mathbb{N}]^\infty$ and a sequence $(g_k)_{k\in L}$ in $\mathrm{co}(\W)$ with $\mathrm{supp}(g_k) \subset \mathrm{supp}(f_k)$ for all $k\in L$, so that for all $j_0\in\mathbb{N}$ and all $F\subset L$ so that $(f_k)_{k\in F}$ is $\mathcal{S}_{n_{j_0}}$-admissible we have that
\[\left\|\sum_{k\in F}\left(f_k - \frac{1}{2}g_k\right)\right\| \leq m_{j_0} + \varepsilon.\]
\end{lem}

\begin{proof}
Let each $f_k = \sum_{r\in F_k}c_r^kf_r^k$, where $f_r^k\in\W$ and $\supp(f_r^k)\subset\supp(f_k)$ for all $r\in F_k$ and $k\in\mathbb{N}$. Without loss of generality we may assume that $\sum_{r\in F_k}c_r^k = 1$ for all $k\in\mathbb{N}$. Define
\[
\begin{split}
\N^{<\infty}_{N} &= \{\vec j = (j_1,\ldots,j_l)\in\N^{<\infty}: m_{j_1}\cdots m_{j_l}\leq N\},\\
F_{\vec j,k} &= \{r\in F_k: \vec w(f_r^k) = \vec j\},\quad \nu_{\vec j,k} = \sum_{r\in F_k}c_r^k\text{ for all }\vec j\in\N^{<\infty}\text{ and }k\in\mathbb{N},\\
F_{N,k} &=\cup_{\vec j\in \N^{<\infty}_N}F_{\vec j,k}\text{ and }G_{N,k} = F_k\setminus G_{N,k},\text{ for all }k,N\in\N.
\end{split}
\]

By passing to a subsequence of $(f_k)_k$ we may assume that for all $\vec j\in\N^{<\omega}$ the limits $\lim_k\nu_{\vec j,k} = \nu_{\vec j}$ exists. Define $\la = \sum_{\vec j\in\N^{<\infty}} \nu_{\vec j}$, which is in $[0,1]$. Fix a sequence of positive real numbers $(\e_i)_i$, with $\sum_i\e_i<\e$, and recursively pick strictly increasing sequences $(k_i)_i$ and $(N_i)_i$ so that the following are satisfied:
\begin{subequations}
\begin{align}
&\left|\la - \sum_{\vec j \in \N^{<\infty}_{N_i}}\nu_{\vec j}\right| < \e_i/3\text{ and if }i>1\text{ then }N_{i} > \max\supp(f_{k_{i-1}}),\label{choose N}\\
\label{close enuffb}
&\sum_{\vec j\in \N^{<\infty}_{N_i}}\left|\nu_{\vec j,k_i} - \nu_{\vec j}\right|< \e_i/3.
\end{align}
Define then for each $i\in\mathbb{N}$ the number $\mu_i = \sum_{r\in G_{N_i,k_i}} c_r^{k_i}$ and note that \eqref{choose N} and \eqref{close enuffb} yield
\begin{equation}
\label{close enuffc}
\begin{split}
\left|\mu_i - (1-\la)\right|&= \left|\la - \sum_{\vec j\in\N_{N_i}^{<\infty}}\nu_{\vec j,k_i}\right|\\
&\leq \sum_{\vec j\in \N^{<\infty}_{N_i}}\left|  \nu_{\vec j} - \nu_{\vec j,k_i}\right| +\!\!\!\!\!\! \sum_{\vec j\in\N^{<\infty}\setminus \N^{<\infty}_{N_i}}\nu_{\vec j}\\
&< \frac{2\e_i}{3}.
\end{split}
\end{equation}
\end{subequations}

For each $i\in\N$, using the convection 1/0 = 0, define
\[f_{\vec j,k_i} = \sum_{r\in F_{\vec j,{k_i}}}\frac{c_r^{k_i}}{\nu_{\vec j,k_i}}f_r^{k_i}, \text{ for }\vec j\in\N^{<\infty}_{N_i}, \text{ and }f_{\mathrm{iw},k_i} = \sum_{r\in G_{N_i,k_i}}\frac{c_r^{k_i}}{\mu_i}f_r^{k_i}.\]
Clearly, all the above functionals are in $\mathrm{co}(\W)$ and a quick inspection reveals that
\begin{equation}
\label{true form}
f_{k_i} = \sum_{\vec j\in\N^{<\infty}_{N_i}}\nu_{\vec j,k_i}f_{\vec j,k_i} + \mu_i f_{\mathrm{iw},k_i}, \text{ with }\sum_{\vec j\in\N^{<\infty}_{N_i}}\nu_{\vec j,k_i} + \mu_i = 1.
\end{equation}
By \eqref{choose N} we observe that if $j_0\in\N$ and $F\subset\mathbb{N}$ is such that $(f_{k_i})_{i\in F}$ is $\mathcal{S}_{n_{j_0}}$ admissible, then by Lemma \ref{vfg covnex} we have that
\begin{equation}
\label{vfg part sums}
\frac{1}{m_{j_0}}\sum_{i\in F}f_{\mathrm{iw},k_i}\in\mathrm{co}(\W).
\end{equation}

In the next step, for each $i\in\mathbb{N}$ and $\vec j\in\N^{<\infty}_{N_i}$, if $\vec j = (j_1,\ldots,j_l)$, write for each $r\in F_{\vec j,k_i}$
\[f_r^{k_i} = \frac{1}{m_{j_1}\cdots m_{j_l}}\sum_{t=1}^{d_r^{k_i}} h_t^{r,i},\]
with $(h_t^{r,i})_{t=1}^{d_r^{k_i}}$ being $\mathcal{S}_{n_{j_1}+\cdots+n_{j_l}}$-admissible and very fast growing, and define $g_r^{k_i} = \frac{2}{m_{j_1}\cdots m_{j_l}}h_1^{r,i}$, which is in $\mathrm{co}(\W)$. Define for each $i\in\mathbb{N}$ and $\vec j\in\N^{<\infty}_{N_i}$ the functional
\[g_{\vec j, k_i} = \sum_{r\in F_{\vec j,{k_i}}}\frac{c_r^{k_i}}{\nu_{\vec j,k_i}}g_r^{k_i},\]
which is in $\mathrm{co}(\W)$ and make the following crucial observations:
\begin{equation}
\label{satisfies same weight lemma}
\begin{array}{c}
f_{\vec j,k_i} - \frac{1}{2}g_{\vec j,k_i} = \sum_{r\in F_{\vec j,{k_i}}}\frac{c_r^{k_i}}{\nu_{\vec j,k_i}}\left(f_r^{k_i} - \frac{1}{2}g_r^{k_i}\right),\\
f_r^{k_i} - \frac{1}{2}g_r^{k_i} = \frac{1}{m_{j_1}\cdots m_{j_l}}\sum_{t=2}^{d_r^{k_i}} h_t^{r,i},\\
\mbox{with $(h_t^{r,i})_{t=2}^{d_r^{k_i}}$  $\mathcal{S}_{n_{j_1}+\cdots+n_{j_l}}$-admissible and very fast growing so that}\\
\min\{w(h_t^{r,i}):2\leq r\leq d_r^{k_i}\} > \min\supp(f_{k_i}).
\end{array}
\end{equation}
Now, Lemma \ref{same weight covnex} and \eqref{satisfies same weight lemma} yield that if we fix $\vec j\in\mathbb{N}^{<\infty}$ then we can deduce that if $j_0\in\mathbb{N}$ and $F\subset\mathbb{N}$ is such that $(f_{k_i})_{i\in F}$ is $\mathcal{S}_{n_{j_0}}$-admissible then, if $F_{\vec j} = \{i: \vec j\in \mathbb{N}^{<\infty}_{N_i}\}$, we have that
\begin{equation}
\label{same weight stuff inside}
\frac{1}{m_{j_0}}\sum_{i\in F_{\vec j}}\left(f_{\vec j,k_i} - \frac{1}{2}g_{\vec j,k_i}\right)\in\mathrm{co}(\W).
\end{equation}
Once we made this observation we set for all $i\in\mathbb{N}$
\[g_{k_i} = \sum_{\vec j\in\N^{<\infty}_{N_i}}\nu_{\vec j}g_{\vec j,k_i},\]
which is in $\mathrm{co}(\W)$ and $\supp(g_{k_i})\subset\supp(f_{k_i})$.

We next wish to show that the conclusion is satisfied for $(g_{k_i})_{i\in\mathbb{N}}$. That is, if $j_0\in\mathbb{N}$ and $(f_{k_i})_{i\in F}$ is $\mathcal{S}_{n_{j_0}}$-admissible, then
\[\left\|\sum_{i\in F}\left(f_{k_i} - \frac{1}{2}g_{k_i}\right)\right\| \leq m_{j_0} + \e.\]
Define for each $i\in\N$ the functional
\[\tilde f_{k_i} = \sum_{\vec j\in\N^{<\infty}_{N_i}}\nu_{\vec j}f_{\vec j,k_i} + (1-\la) f_{\mathrm{iw},k_i},\]
which is in $\mathrm{co}(W)$. By \eqref{close enuffb}, \eqref{close enuffc}, and \eqref{true form} we obtain $\|f_{k_i} - \tilde f_{k_i}\| <\e_i$. By this, it is now sufficient to prove that, if $(f_{k_i})_{i\in F}$ is $\mathcal{S}_{n_{j_0}}$-admissible, then
\begin{equation}
\label{the goal}
f = \frac{1}{m_{j_0}}\sum_{i\in F}\left(\tilde f_{k_i} - 
\frac{1}{2}g_{k_i}\right)\in\mathrm{co}(\W)
\end{equation}
because this will imply $\|f\| \leq 1$. The conclusion will then follow from a simple application of the triangle inequality. We are now ready to dissect $f$. Set $N_0 = \max_{i\in F} N_i$ and for each $\vec j\in\N^{<\infty}$ $F_j = \{i\in F: \vec j\in\N^{<\infty}_{N_i}\}$. Write
\begin{align*}
f& = \frac{1}{m_{j_0}}\sum_{i\in F}\left(\left(\sum_{\vec j\in\N^{<\infty}_{N_i}}\nu_{\vec j}\left(f_{\vec j,k_i} - \frac{1}{2}g_{\vec j,k_i}\right)\right)+(1 - \la)f_{\mathrm{iw},k_i}\right)\\
&=\left(\sum_{\vec j\in\N_{N_0}^{<\infty}}\nu_j\left(\frac{1}{m_{j_0}}\sum_{i\in F_{\vec j}}\left(f_{\vec j,k_i} - \frac{1}{2}g_{\vec j,k_i}\right)\right)\right) + (1-\la)\frac{1}{m_{j_0}}\sum_{i\in F}f_{\mathrm{iw},k_i}.
\end{align*}
Finally, by \eqref{vfg part sums} and \eqref{same weight stuff inside}, $f$ is a convex combination of elements of $\mathrm{co}(\W)$ and hence it is in $\mathrm{co}(\W)$.
\end{proof}

\begin{prop}
\label{subseq mT dual}
Let $(f_k)_k$ be a block sequence in the unit ball of $\X^*$. Then for any $\e>0$ there exists $L\in[\N]^\infty$ so that for any $j_0\in\mathbb{N}$ and $F\subset \N$ with $(f_k)_{k\in F}$ being $\mathcal{S}_{n_{j_0}}$-admissible we have
\[\left\|\sum_{k\in F}f_k\right\| \leq 2m_{j_0}+\e.\]
\end{prop}

\begin{proof}
By reflexivity we have that the unit ball of $\X^*$ is the closed convex hull of $\W$. Actually, a compactness argument yields that every finitely supported vector in the unit ball of $\X^*$ must be in $\mathrm{co}(\W)$. Set $(f_k^{(0)})_k = (f_k)_k$ and apply Lemma \ref{getting rid of} inductively to find infinite sets $L_1\supset L_2\supset\cdots\supset L_q\supset\cdots$ and, for each $q\in\N$, $(f_k^{(q)})_{k\in L_q}$ in $\mathrm{co}(\W)$ so that for all $j_0\in\N$ and $F\subset L_q$ with $(f_k^{(q-1)})_{k\in L_q}$ being $\mathcal{S}_{n_{j_0}}$ we have that
\[\left\|\sum_{k\in F}\left(f_k^{(q-1)}-\frac{1}{2}f_k^{(q)}\right)\right\| \leq m_{j_0} + \frac{\e}{4}.\]
Pick $q_0\in\N$ with $1/2^{q_0-1} < \e/2$ and then pick an infinite subset of $L_{q_0}$ $L = \{\ell_i:i\in\N\}$ so that for all $q\geq q_0$ and $i\geq q$ we have $\ell_i\in L_q$.  Let now $j_0\in\mathbb{N}$ and $F\subset L$ so that $(f_k^{(0)})_{k\in F}$ is $\mathcal{S}_{n_{j_0}}$-admissible. If $F = \{k_1,\ldots,k_N\}$, define for $q = 0,1,\ldots,q_0$ the set $F_q = \{k_1,\ldots,k_N\}$ and for $q = q_0+1,\ldots,n$ the set $F_q = \{k_q,\ldots,k_N\}$. Observe that $F_q\subset L_q$ and $(f_k^{(q-1)})_{k\in F_q}$ is $\mathcal{S}_{n_{j_0}}$-admissible. Then,
\[
\begin{split}
\left\|\sum_{k\in F}f_k^{(0)}\right\| &= \left\|\sum_{k\in F}f_k^{(0)} + \sum_{q=1}^N\frac{1}{2^q}\sum_{k\in F_q}(f_k^{(q)} - f_k^{(q)})\right\|\\
&= \left\|\sum_{q=1}^N\left(\frac{1}{2^{q-1}}\sum_{k\in F_{q-1}}f_k^{(q-1)}-\frac{1}{2^q}\sum_{k\in F_q}f_k^{(q)}\right) + \frac{1}{2^N}f^{(N)}_{k_N}\right\|\\
&\leq \sum_{q=1}^{q_0}\frac{1}{2^{q-1}}\left\|\sum_{r=1}^Nf^{(q-1)}_{k_r} - \frac{1}{2}f^{(q)}_{k_r}\right\|\\
&+ \sum_{q=q_0+1}^N\frac{1}{2^{q-1}}\left(\left\|f^{(q-1)}_{k_{q-1}}\right\| + \left\|\sum_{r=q}^Nf^{(q-1)}_{k_r} - \frac{1}{2}f^{(q)}_{k_r}\right\|\right) + \frac{1}{2^N}\left\|f^{(N)}_{k_N}\right\| \\
&\leq \sum_{q=1}^N\frac{1}{2^{q-1}}\left(m_{j_0} + \frac{\e}{4}\right) + \sum_{q=q_0}^N\frac{1}{2^{q}} \leq 2m_{j_0} + \e.
\end{split}
\]
\end{proof}

\begin{cor}
Every normalized block sequence in $\X^*$ has a subsequence that generates a spreading model 4-equivalent to the unit vector basis of $c_0$.
\end{cor}

\begin{proof}
 Let $(f_k)_k$ be a normalized block sequence in the unit ball of $\X^*$ and apply Proposition \ref{subseq mT dual}, for some $\e>0$, and relabel to assume that conclusion holds for the whole sequence. By 1-unconditionality we deduce that for any $F\subset \mathbb{N}$ so that $(f_k)_{k\in F}$ is $\mathcal{S}_{n_1}$-admissible we have that $(f_k)_{k\in F}$ is $(2m_1+\e)$-equivalent to the unit vector basis of $c_0$.  Recall that $m_1 = 2$ and $n_1 = 1$.
\end{proof}

Reflexivity of $\X$, the above stated corollary, and Proposition \ref{unique 1 or 0 am} yield the next result.
\begin{cor}
The space $\X^*$ is asymptotically symmetric.
\end{cor}

For $n\in\N$ we shall say that a finite block sequence $(f_k)_{k=1}^d$ in $\X^*$ is maximally $\mathcal{S}_n$-admissible if $\{\min\supp(f_k):1\leq k\leq d\}$ is a maximal $\mathcal{S}_n$-set.

\begin{prop}
\label{very good c0 vectors}
Let $Y$ be a block subspace of $\X^*$. Then for every $n\in\N$ and $\de>0$ there exists a sequence $(f_k)_{k=1}^d$ that is maximally $\mathcal{S}_n$-admissible with $\|f_k\| \geq 1$ for $k=1,\ldots,d$ and $\|\sum_{k=1}^df_k\| \leq 1+\de$.
\end{prop}

\begin{proof}
The proof goes along the lines of the proof of Proposition \ref{very good ell1 vectors}. Start with a normalized sequence $(f_i)_i$, to which we apply Proposition \ref{subseq mT dual}, and assume that the conclusion fails in the linear span of this sequence. We can then find for every $j\in\N$ with $j\geq n$ an integer $d_j$ with $n_j - n\leq d_jn\leq n_j$ and an $F_j$ so that $(f_k)_{k\in F_j}$  is maximally $\mathcal{S}_{d_jn}$-admissible with
\[2m_j + \e\geq\left\|\sum_{i\in F_j}f_i\right\| \geq \left(1+\de\right)^{d_j+1} \geq \left(1+\de\right)^{n_j/n}.\]
This implies that $\limsup_j((1+\de)^{1/n})^{n_j}/m_j\leq 2$ which contradicts the first property of the sequences $(m_j)_j$, $(n_j)_j$ (see Section \ref{definition section}).
\end{proof}

\begin{cor}
\label{dual RIS exists}
Let $Y$ be a block subspace of $\X^*$ and let $C>1$. Then there exist a block sequence $(y_n^*)_n$ in $Y$ and a block sequence $(y_n)_n$ in $\X$ so that the following hold.
\begin{itemize}
\item[(i)] $1\leq\|y_n\|$ and $\|y_n^*\| \leq C$ for all $n\in\N$,
\item[(ii)] $\supp(y_n) = \supp(y_n^*)$ and $y_n^*(y_n) = 1$, and
\item[(iii)] $(y_n)_n$ is a $C$-RIS.
\end{itemize}
\end{cor}

\begin{proof}
Fix $C>1$ and apply Lemma \ref{very good c0 vectors} to find a block sequence $(y_n^*)_n$ so that that for all $n\in\N$ we have $\|w^*_n\| \leq (1+\sqrt C)/2$, $\min\supp(w_n^*) \geq (6n)/(\sqrt C - 1)$, and $y_n$ is of the form $w^*_n = \sum_{i \in F_n}f_i$ with $f_i$ in $Y$, $\|f_i\| \geq 1$, for all $i\in\N$, $(f_i)_{i\in\mathbb{N}}$ is maximally $\mathcal{S}_n$-admissible. Pick for each $n\in\N$ and $i\in\N$ a normalized vector $x_i$ with $\supp(x_i)\subset\supp(f_i)$ and $f_i(x_i) \geq 1$. For each $n\in\N$ we may perturb each vector $x_i$ to assume that $\supp(x_i) = \supp(f_i)$. By scaling we can ensure that all the aforementioned properties are retained, only perhaps increasing the upper bound of $\|w_n^*\|$ to $\|w_n^*\|\leq \sqrt C$.

Because, for each $n\in\N$, $(x_i)_{i\in F_n}$ is maximally $\mathcal{S}_n$-supported, by \cite[Proposition 2.3]{AT}, we can find coefficients $(c_i)_{i\in F_n}$ so that the vector $w_n = \sum_{i\in F_n}c_ix_i$ is a $(n,\e)$-s.c.c. with $\e \leq 3/\min\supp(w_n^*) \leq (\sqrt C - 1)/(2n)$. By Proposition 6.5 we have that for every $f\in\W$, with $w(f) = (j_1,\ldots,j_l)$ and $n_1+\cdots+n_l < n$ the estimate $|f(w_n)| \leq \sqrt{C}/w(f)$. It follows that $(w_n)_n$ has a subsequence $(w_{k_n})_n$ that is a $\sqrt{C}$-RIS.

Note that $w_{k_n}^*(w_{k_n}) = \sum_{i\in F_{k_n}}c_if_i(x_i) = 1$, hence $1\geq \|w_{k_n}\| \geq 1/\|w_{k_n}^*\| \geq 1/\sqrt{C}$. Thus, the sequence $(y_n)_n = (\sqrt{C}w_{k_n})_n$ is a  $C$-RIS with $\|y_n\| \geq 1$ for all $n\in\N$ and the sequence $(y_n^*)_n = (w_{k_n}^*/\sqrt{C})_n$ satisfies $\|y_n^*\| \leq C$ and $y_n^*(y_n) = 1$ for all $n\in\N$.
\end{proof}

\begin{thm}
\label{dual omega joint}
Let $Y$ be a block subspace of $\X^*$. Then $Y$ contains an array of normalized block sequences $(f_j^{(i)})_j$, $i\in\mathbb{N}$, that generates an asymptotic model equivalent to the unit vector basis of $\ell_1$.
\end{thm}

\begin{proof}
The proof of this result follows the proof of Proposition \ref{omega joint spreading models}. Fixing $\e>0$, choose $C>1$ and a sequence $(j_i)_i$ as in the aforementioned proof. Apply Corollary \ref{dual RIS exists} to find a $C$-RIS $(y_s)_s$ and a sequence $(y_s^*)_s$ in $Y$ with properties (i), (ii), and (iii) in the statement of that result. Pass to common subsequences, by applying Proposition \ref{subseq mT dual}, so that for any $j_0\in\N$ and any $F\subset \N$ so that $(y_s)_{s\in F}$ is $\mathcal{S}_{n_{j_0}}$-admissible we have $\|\sum_{s\in F}y_s\| \leq 3m_{j_0}$. 

Following the proof of Proposition \ref{omega joint spreading models} define an array of block sequences $(x^{(i)}_j)_j$, $i\in\N$, that satisfies \eqref{omega equation}, so that each vector $x^{(i)}_j$ is of the form $x^{(i)}_j = m_{j_i}\sum_{s\in F_j^{(i)}}c_s^{i,j}y_s$, with $(y_s)_{s\in F_j^{(i)}}$ $\mathcal{S}_{n_{j_i}-1}$-admissible and $\sum_{s\in F_j^{(i)}}c_s^{i,j} = 1$. Also, the sets $(F_j^{(i)})_j$, $i\in\N$ are all pairwise disjoint. If we then define $f_j^{(i)} = \sum_{s\in F_j^{(i)}}y_i$, for $i,j\in\N$, we have that $\|f_j^{(i)}\| \leq 3$, and $f_j^{(i)}(x_j^{(i)}) = 1$, and $f_j^{(i)}(x_{j'}^{(i')}) = 0$ if $(i,j)\neq(i',j')$. For every $n\leq j_1<\cdots<j_n$ the sequence $(x_{j_i}^{(i)})$ has a $(1+\e)$-upper $c_0$-estimate which yields that $(f_{j_i}^{(i)})$ has a $1/(1+\e)$-lower $\ell_1$ estimate and therefore it is $3(1+\e)$-equivalent to the unit vector basis of $\ell_1$.
\end{proof}

\begin{rem}
A slightly a more careful version of the above proof yields that in every block subspace $Y$ of $\X^*$, for every $m\in\N$ one can find a array $(f_j^{(i)})_j$, $1\leq i\leq m$ that generates a joint spreading model  3-equivalent to the unit vector basis of $\ell_1^m(c_0)$. It is not clear what the asymptotic spaces of $Y$ are. Although $\tilde K(Y) = \{\infty\}$ all we know about the set $K(Y)$ is $\{1,\infty\} \subset K(Y)$.
\end{rem}

\section{The space $\tilde{X}_{\mathbf{iw}}$}
\label{the other space}
The purpose of this section is to simplify the definition of the space $\X$ to obtain a new space $\Xt$. This new space also has the property that every normalized block sequence in $\Xt$ has a subsequence generating a spreading model equivalent to the unit vector basis of $\ell_1$ without containing a subspace where all spreading models of normalized block sequences are uniformly equivalent to $\ell_1$.

\subsection{Definition of $\Xt$}
We simplify the definition of the norming set $\W$ of $\X$ by only considering functionals of the form $(1/m_j)\sum_{q=1}^d f_j$.

\begin{defn}
Let $\Wt$ be the smallest subset of $c_{00}(\N)$ that satisfies the following to conditions.
\begin{itemize}

\item[(i)] $\pm e_i^*$ is in $\Wt$ for all $i\in\N$ and

\item[(ii)] for every $j\in\N$ and every $\mathcal{S}_{n_j}$ very fast growing sequence of weighted functionals $(f_q)_{q=1}^d$ in $\Wt$ the functional $$f = \frac{1}{m_j}\sum_{q=1}^df_q$$ is in $\Wt$.

\end{itemize}
We define a norm on $c_{00}(\N)$ given by $\iii{x} = \sup\{f(x): x\in \Wt\}$ and we set $\Xt$ to be the completion of $(c_{00}(\N), \iii{\cdot})$.
\end{defn}

\begin{defn}
For each $j\in\N$ we define the norm $\|\cdot\|_{\ell_1,j}$ on $\ell_1(\N)$ given by
\begin{equation}
\left\|\sum_{k=1}^\infty a_ke_k\right\|_{\ell_1,j} = \max\left\{\max_k |a_k|,\frac{m_{j}}{m_{j+1}}\sum_{k=1}^\infty|a_k|\right\}.
\end{equation}
\end{defn}

Clearly, this norm is equivalent to the usual norm of $\ell_1$, however this equivalence is not uniform in $j\in\N$. This can be seen by taking, e.g., the vector $x_j = \sum_{k=1}^{m_{j+1}}e_k$ in which case $\|x_j\|_{\ell_1,j} = m_j$ whereas $\|x_j\|_{\ell_1} = m_{j+1}$. We will see that every block subspace of $\Xt$ for every $j\in\N$ contains a block sequence that generates a spreading model isometrically equivalent to the unit vector basis of $\ell_1(\N)$ endowed with $\|\cdot\|_{\ell_1,j}$.

\subsection{The auxiliary space for $\Xt$}
The auxiliary spaces are almost identical as those for the space $\X$, the difference being the lack of the factors $1/2^l$.
\begin{defn}
For $N\in\N$ let $\widetilde W_\mathrm{aux}^N$ be the smallest subset of $c_{00}(\N)$ that satisfies the following to conditions.
\begin{itemize}

\item[(i)] $\pm e_i^*$ is in $W_\mathrm{aux}$ for all $i\in\N$ and

\item[(ii)] for every $j\in\N$ and every $\mathcal{S}_{n_j}\ast\mathcal{A}_3$ admissible sequence of $N$-sufficiently large auxiliary weighted functionals $(f_q)_{q=1}^d$ in $\widetilde W_\mathrm{aux}$ the functional $$f = \frac{1}{m_j}\sum_{q=1}^df_q$$ is in $\widetilde W_\mathrm{aux}$.

\end{itemize}
We define a norm $\iii{\cdot}_{\mathrm{aux},N}$ on $c_{00}(\N)$ by defining for all $x\in c_{00}(\N)$ the quantity  $\iii{x}_{\mathrm{aux},N} = \sup\{f(x): f\in W_\mathrm{aux}^N\}$.
\end{defn}

\begin{lem}
\label{upper auxiliary second space}
Let $n$, $j_0, N\in\N$ with $N\geq 2m_{j_0}$,  $(\e_k)_{k=1}^n$ be a sequence of real numbers with $0<\e_k < 1/(6m_{j_0})$ for $1\leq k\leq n$ and $(x_{k})_{k=1}^n$ be vectors in $c_{00}(\N)$ so that for each $1\leq k\leq n$ the vector $x_k$ is of the form
\begin{equation}
x_{k} = m_{j_0}\tilde x_{k},\text{ where } \tilde x_{k} = \sum_{r\in F_{k}}c_r^{k}e_r \text{ is a } (n_{j_0},\e_k) \text{ basic s.c.c.}
\end{equation}
Then, for any scalars $(a_{k})_{k=1}^n$ and $f\in\widetilde W_\mathrm{aux}^N$, we have
\begin{equation}
\label{upper auxiliary second space eq}
\left|f\left(\sum_{k=1}^na_{k}x_{k}\right)\right|\leq (1+\de)\max\left\{\max_{1\leq k\leq n}|a_k|,\frac{m_{j_0}}{m_{j_0+1}}\sum_{k=1}^n|a_k|\right\},
\end{equation}
for any $\de$ satisfying
\begin{equation}
\label{this long delta second space}
 \de \geq \max\left\{\frac{2m_{j_0+1}}{N},6\sum_{k=2}^n\max\supp(x_{k-1})\e_k,6m_{j_0}\sum_{k=2}^n\e_k\right\}.
\end{equation}
\end{lem}

\begin{proof}
We perform an induction on $m=0,1,\ldots$ to show that for all $f\in \widetilde W^N_m$ and for all $1\leq k\leq n$ we have $|f(x_k)|\leq 1$ as well as that \eqref{upper auxiliary second space eq} holds  for $f$. The step $m=0$ is trivial so let $m\in\N$, assume that the inductive assumption holds for all $f\in\widetilde W_m^N$ and let $f\in\widetilde W_{m+1}^N\setminus\widetilde W_{m}^N$. Let $f = (1/m_j)\sum_{q=1}^df_q$ where $(f_q)_{q=1}^d$ is $\mathcal{S}_{n_j}$ admissible and $N$ sufficiently large. If $j > j_0$ then an elementary calculation yields $|f(x_k)| \leq m_{j_0}/m_{j_0+1}$ for $1\leq k\leq n$ and hence \eqref{upper auxiliary second space eq} easily follows. Therefore, we may assume that $j\leq j_0$.

Set $M_k = \max\supp(x_k)$ for $1\leq k\leq n$, $k_0 = \min\{k: \min\supp(f)\leq M_k\}$, if such a $k_0$ exists, and set $q_0 = \min\{q:\max\supp(f_q)\geq\min\supp(x_{k_0})\}$. For simplicity let us assume $q_0 = 1$. Set $\tilde f = (1/m_j)\sum_{q=2}^df_q$, $G = \{2\leq q\leq d: f_q = \pm e_i^* \text{ for some } i\in\N\}$, $D = \{2,\ldots,d\}\setminus G$, and
\begin{equation*}
g_1 = \frac{1}{m_j}\sum_{q\in G}f_q,\quad g_2 = \frac{1}{m_j}\sum_{q\in D}f_q.
\end{equation*}
As the sequence $(f_q)_{q=1}^d$ is $N$-sufficiently large we obtain $w(f_q) \geq N$ for all $q\in D$ which easily implies
\begin{equation}
\label{estimate g2 in this proof}
\begin{split}
\left|g_2\left(\sum_{k=k_0+1}^na_kx_k\right)\right| &\leq \frac{m_{j_0}}{m_j N}\sum_{k=k_0+1}^n|a_k| \leq \left(\frac{m_{j_0+1}}{2 N}\right)\frac{m_{j_0}}{m_{j_0+1}}\sum_{k=k_0+1}^n|a_k|\\
&\leq \frac{\de}{4}\frac{m_{j_0}}{m_{j_0+1}}\sum_{k=k_0+1}^n|a_k|
\end{split}
\end{equation}

We now estimate the quantity $g_1(\sum_{k = k_0+1}^na_kx_k)$ and we distinguish cases depending on the relation of $m_j$ and $m_{j_0}$. We first treat the case $j = j_0$. As $\{\min\supp (f_q):1\leq q\leq d\}$ is in $\mathcal{S}_{n_{j_0}}\ast\mathcal{A}_3$ it follows that $l\leq M_{k_0}$ and there are $G_1<\cdots<G_l$ in $\mathcal{S}_{n_{j_0}-1}\ast\mathcal{A}_3$ so that $G = \cup_{p=1}^lG_p$. If we set $h_p = (1/m_{j_0})\sum_{s\in G_p}f_s$ then for $1\leq p\leq l$ and $k_0 < k\leq n$ we have $|h_p(x_k)| \leq (1/m_{j_0})3\e_k$ which yields
\begin{equation*}
\begin{split}
\left|g_1\left(\sum_{k>k_0}a_kx_k\right)\right| &\leq \frac{1}{m_{j_0}}\sum_{p=1}^l\left|h_p\left(\sum_{k>k_0}a_kx_k\right)\right|\leq  \frac{M_{k_0}}{m_{j_0}}\sum_{k>k_0}3\e_k\max_{k_0< k\leq n}|a_k|\\
&\leq \left(\frac{3}{2}\sum_{k=2}^nM_{k-1}\e_k\right)\max_{1\leq k\leq n}|a_k|.
\end{split}
\end{equation*}
In the second case $j<j_0$ and we use a simpler argument to show that
$$\left|g_1\left(\sum_{k=k_0+1}^na_kx_k\right)\right| \leq \frac{m_{j_0}}{m_j}\sum_{k=2}^n3\e_k\max_{k_0<k\leq n}|a_k| \leq \left(\frac{3m_{j_0}}{2}\sum_{k=2}^n\e_k\right)\max_{1\leq k\leq n}|a_k|.$$
We conclude that in either case we have
\begin{equation}
\label{this specific part when weight is at most that much}
\left|g_1\left(\sum_{k=k_0+1}^na_kx_k\right)\right| \leq \frac{\de}{4}\max_{1\leq k\leq n}|a_k|.
\end{equation}
Before showing that $f$ satisfies \eqref{upper auxiliary second space eq} we quickly show that $|f(x_k)| \leq 1$ for $1\leq k\leq n$ (there is a more classical proof that depends on the properties of the sequences $(m_j)_j$ and $(n_j)_j$ however the constraints make the proof faster). If $j = j_0$ this is easy. Otherwise $j<j_0$ and arguments very similar to those above yield
\begin{equation*}
\begin{split}
|f(x_k)| &\leq \frac{1}{m_j}|f_1(x_k)| + |g_1(x_k)| + |g_2(x_k)| \leq \frac{1}{m_j} + \frac{m_{j_0}}{m_j}3\e_k + \frac{m_{j_0}}{m_jN}\\
&\leq \frac{1}{2} + \frac{1}{4} + \frac{1}{4} = 1.
\end{split}
\end{equation*}

Set $$L = \max\left\{\max_{1\leq k\leq n}|a_k|,\frac{m_{j_0}}{m_{j_0+1}}\sum_{k=1}^n|a_k|\right\}.$$
We now distinguish cases concerning the support of $f_1$ in relation to the support of $x_{k_0}$. If $\max\supp(f_1) > \max\supp(x_{k_0})$ then
\begin{equation*}
\begin{split}
\left|f\left(\sum_{k=1}^na_kx_k\right)\right| & \leq \frac{1}{m_{j_0}}\left|f_1\left(\sum_{k=1}^na_kx_k\right)\right| + \left|\left(g_1 + g_2\right)\left(\sum_{k=k_0+1}^na_kx_k\right)\right|\\
&\leq \frac{1}{m_{j_0}}(1+\de)L + \frac{2\de}{4}L \leq \left[\frac{1}{2} + \left(\frac{1}{2} + \frac{2}{4}\right)\de\right]L\leq (1+\de)L.
\end{split}
\end{equation*}
If $\max\supp(f_1) \leq \max\supp(x_{k_0})$ then
\begin{equation*}
\begin{split}
\left|f\left(\sum_{k=1}^na_kx_k\right)\right| & \leq \left|f\left(a_{k_0}x_{k_0}\right)\right| + \left|\left(g_1 + g_2\right)\left(\sum_{k=k_0+1}^na_kx_k\right)\right|\\
&\leq L + \frac{2\de}{4}L \leq \left(1 + \frac{2\de}{4}\right)L\leq (1+\de)L.
\end{split}
\end{equation*}
The inductive step is complete and so is the proof.
\end{proof}

\subsection{The spreading models of $\Xt$}
We observe that all spreading models of normalized block sequences in $\Xt$ are equivalent to $\ell_1$ and we construct in every subspaces a block sequence that generates a spreading model equivalent to $\ell_1$ but with arbitrarily bad isomorphism constant.

\begin{prop}
\label{at least you have ell1 spreading model}
Let $(x_i)_i$ be a normalized block sequence in $\Xt$. Then there exist $L\in[\N]^\infty$ of $(x_i)_i$ and $K_0\in\N\cup\{0\}$ so that for every $j,k\in\N$ with $k\leq n_j - K_0$, every $F\subset L$ with $(x_i)_{i\in F}$ $\mathcal{S}_k$ admissible, and every scalars $(c_i)_i\in F$ we have
\begin{equation*}
\left\|\sum_{i\in F}c_ix_i\right\| \geq \frac{1}{m_j}\sum_{i\in F}|c_i|. 
\end{equation*}
In particular, every normalized block sequence in $\Xt$ has a subsequence that generates a spreading model equivalent to the unit vector basis of $\ell_1$. 
\end{prop}

\begin{proof}
Take a sequence of functionals $(f_i)_i$ in $\W$ with $\ran(f_i)\subset \ran(x_i)$ and $f_i(x_i) = 1$ for all $i\in\N$. , namely the one in which $\limsup_kw(f_k)$ is finite and the one in which it is infinite. 

We shall only treat the first case as the second one is simpler and it follows for $K_0 = 0$. By passing to an infinite subset of $\N$ and relabeling there is $j_0\in\N$ with $w(f_i) = m_{j_0}$ for all $i\in\N$. Define $K_0 = n_{j_0} $. Write each $f_i$ as
\[f_i = \frac{1}{m_{j_0}}\sum_{q=1}^{d_i}f_q^i\]
with $(f_q^i)_{q=1}^{d_i}$ being $\mathcal{S}_{K_0}$-admissible and very fast growing. Arguing as in \eqref{uniform ell1 eq 1} it follows that for all $i$ we have $\sum_{q=2}^{d_i}f_q^i(x_i) \geq (1/2)m_{j_0}\geq 1$ and passing to a subsequence and relabeling we have that $((f_q^i)_{q=2}^{d_i})_i$ is very fast growing.

We can conclude that for any $j,k\in\N$ and any $F\subset \N$ so that $(x_i)_{i\in F}$ is $\mathcal{S}_k$-admissible with $k\leq n_j - K_0$, the sequence $((f_q^i)_{q=2}^{d_i})_{q\in F}$ is $\mathcal{S}_{n_j}$ admissible because $\mathcal{S}_k\ast\mathcal{S}_{K_0} = \mathcal{S}_{k+K_0}$ and $k + K_0 \leq n_j$. Hence, $f_F = (1/m_j)\sum_{i\in F}\sum_{q=2}^{d_i}f_q^i$ is in $\Wt$. This means that for any scalars $(c_i)_{i\in F}$ we have
\begin{equation}
 \left\|\sum_{i\in F}c_ix_i\right\| = \left\|\sum_{i\in F}|c_i|x_i\right\| \geq f_F\left(\sum_{i\in F}|c_i|x_i\right) \geq \frac{1}{m_j}\sum_{i\in F}|c_i|.
\end{equation}
\end{proof}

\begin{prop}
\label{bad ell1 spreading models all over the place}
Let $Y$ be a block subspace of $\Xt$. Then for every $j_0\in\N$ there exists a sequence $(x_k)_k$ in $Y$ that generates a spreading model isometrically equivalent to the unit vector basis of $(\ell_1,\|\cdot\|_{\ell_1,j_0})$.
\end{prop}

Before proving the above statement we point out that RIS sequences in $\Xt$ are defined identically as in Definition \ref{definition ris} and Proposition \ref{basic inequality} is also true by taking the set $\widetilde W_\mathrm{aux}^N$. Furthermore all results of subsection \ref{section ris existence} are true for the space $\Xt$ and the proofs are very similar. In particular Corollary \ref{building ris} is true in $\Xt$ and this is proved by using Proposition \ref{at least you have ell1 spreading model}.

\begin{proof}[Proof of Proposition \ref{bad ell1 spreading models all over the place}]
For a sequence of positive numbers $(C_k)_k$ decreasing strictly to one apply Corollary \ref{building ris} to find a sequence $(y_i)_i$ in $Y$ so that for all $k\in\N$ the sequence $(y_i)_{i\geq k}$ is $(C_k,(j_i)_{i\geq k})$-RIS with $\|y_i\|\geq 1$ for all $i\in\N$ (this is possible via a minor modification of the proof of Corollary \ref{building ris} in which $\de$ is replaced by $\de_i$). Inductively build a sequence $(x_k)_k$ so that for all $k\in\N$ the vector $x_k$ is of the form $x_k = m_{j_0}\tilde x_k$ where $\tilde x_k = \sum_{i\in F_k}c_i^ky_i$ a $(n_{j_0},\e_k/2)$ s.c.c. with $\e_{k+1} < (2^k\max\supp(x_k))^{-1}$ for all $k\in\N$. As in the proof of Proposition \ref{omega joint spreading models} we can find for all $k\in\N$ a sequence of very fast growing and $\mathcal{S}_{n_{j_0}}$ admissible functionals $(f_i)_{i\in F_k}$  in $\Wt$ with $\supp(f_i)\subset y_i$ for all $i\in F_k$ so that if $f_k = (1/m_{j_0})\sum_{i\in F_k}f_i\in \Wt$ then $f_k(x_k) = 1$ and so that the sequence $((f_i)_{i\in F_k})_k$ enumerated in the obvious way is very fast growing. We deduce that for all natural numbers $n \leq k_1<\cdots<k_n$ the functionals $((f_i)_{i\in F_{k_l}})_{l=1}^n$ are $\mathcal{S}_{n_{j_0}+1}$ admissible. This means that they are also $\mathcal{S}_{n_{j_0+1}}$ admissible i.e. $f = (1/m_{j_0+1})\sum_{l=1}^n\sum_{i\in F_{k_l}}f_i = (m_{j_0}/m_{j_0+1})\sum_{l=1}^nf_{k_l}$ is in $\Wt$. We conclude that for any scalars $(a_l)_{l=1}^n$ we have
\begin{equation*}
\iii{\sum_{l=1}^na_lx_{k_l}} = \iii{\sum_{l=1}^n|a_l|x_{k_l}} \geq f\left(\sum_{l=1}^n|a_l|x_{k_l}\right) \geq \frac{m_{j_0}}{m_{j_0+1}}\sum_{l=1}^n|a_l|
\end{equation*}
and also 
\begin{equation*}
\iii{\sum_{l=1}^na_lx_{k_l}} = \iii{\sum_{l=1}^n|a_l|x_{k_l}} \geq \max_{1\leq l\leq n}f_{k_l}\left(\sum_{l=1}^n|a_l|x_{k_l}\right) = \max_{1\leq l\leq n}|a_l|.
\end{equation*}

For the upper inequality, Proposition \ref{basic inequality} and Lemma \ref{upper auxiliary second space} imply that there is a null sequence of positive numbers $\de_n$ so that for all natural numbers $n\leq k_1 <\cdots < k_n$ and scalars $(a_l)_{l=1}^n$ we have
\begin{equation*}
\iii{\sum_{l=1}^na_lx_{k_l}}\leq \left(1+\de_n\right)\max\left\{\max_{1\leq l\leq n}|a_l|,\frac{m_{j_0}}{m_{j_0+1}}\sum_{l=1}^n|a_l|\right\}.
\end{equation*}
\end{proof}

\begin{rem}
It can be shown that the space $\Xt$ satisfies the conclusions of Theorem \ref{c0 asmodel} and Corollary \ref{hereditary fbr}. Note also that unlike $\tilde K(\X)$, the set $\widetilde K(\Xt)$ contains $\{1,\infty\}$. It is unclear whether $\widetilde K(\Xt)$ contains any $p$'s in $(1,\infty)$.
\end{rem}

As it was shown in Section \ref{dual section} the space $\X^*$ admits only the unit vector basis of $c_0$ as a spreading model. This is false for the space $\Xt^*$.

\begin{prop}
The space $\Xt^*$ admits spreading models that are not equivalent to the unit vector basis of $c_0$.\end{prop}

\begin{proof}
\cite[Proposition 3.2]{AOST} yields that if a space has the property that every spreading model generated by a normalized weakly sequence in that space is equivalent to the unit vector basis of $c_0$, then there must exist a uniform constant $C$ so that this equivalence is always with constant $C$. We point out that this conclusion only works for the spacial case $p=\infty$ and not for other $p$'s, because the unit vector basis of $c_0$ is the minimum norm with respect to domination. By duality we would obtain that every spreading model generated by a normalized block sequence in $\Xt$ is $C$-equivalent to the unit vector basis of $\ell_1$. This would contradict the statement of Proposition \ref{bad ell1 spreading models all over the place}.
\end{proof}

\end{document}